\documentclass[11pt]{amsart}

\usepackage[mathcal]{euscript}
\usepackage{amssymb, amsfonts, xypic, amsmath}
\usepackage{amscd}
\usepackage[all]{xy}
\usepackage[margin=1in]{geometry}
\newtheorem*{theorem*}{Theorem}
\newtheorem{theorem}{Theorem}[section]
\newtheorem{lemma}[theorem]{Lemma}
\newtheorem{corollary}[theorem]{Corollary}

\newtheorem{proposition}[theorem]{Proposition}

\theoremstyle{definition}
\newtheorem{definition}[theorem]{Definition}

\newtheorem{example}[theorem]{Example}
\newtheorem{remark}[theorem]{Remark}

\numberwithin{equation}{section} \numberwithin{figure}{section}

\numberwithin{equation}{section}

\newtheorem*{ack}{Acknowledgments}

\usepackage{graphicx}
\usepackage{epstopdf}

\author{Manuel Rivera and Mahmoud Zeinalian}

\newcommand{\Addresses}{{
  \bigskip
  \footnotesize

    \textsc{Manuel Rivera, Department of Mathematics, University of Miami, 1365 Memorial Drive, Coral
Gables, FL 33146 and
Departamento de Matem\'aticas, Cinvestav, Av. Instituto Polit\'ecnico Nacional 2508, Col. San Pedro Zacatenco, M\'exico, D.F. CP 07360, M\'exico} \par\nopagebreak
  \textit{E-mail address} \texttt{manuelr@math.miami.edu}

  \medskip
  \medskip

  \textsc{Mahmoud Zeinalian, Department of Mathematics, Lehman College of CUNY, 250 Bedford Park Boulevard West,
Bronx, NY 10468
   }\par\nopagebreak
  \textit{E-mail address} \texttt{mahmoud.zeinalian@lehman.cuny.edu}

}}

\begin{document}

\begin{abstract} We describe several equivalent models for the $\infty$-category of $\infty$-local systems of chain complexes over a space using the framework of quasi-categories. We prove that the given models are equivalent as $\infty$-categories by exploiting the relationship between the differential graded nerve functor and the cobar construction. We use one of these models to calculate the quasi-categorical colimit of an $\infty$-local system in terms of a twisted tensor product. 
\end{abstract}

\title[The colimit of an $\infty$-local system as a twisted tensor product]{The colimit of an $\infty$-local system as a twisted tensor product}
\maketitle

\section{Introduction}
The goal of this paper is to give an explicit model for the homotopy coherent colimit of an $\infty$-local system of chain complexes over a topological space in terms of Brown's twisted tensor product construction. We use the framework of quasi-categories to describe three equivalent $\infty$-categories of $\infty$-local systems and then use one of these models to calculate the desired colimit. We generalize the following classical story to the homotopy coherent setting. Let $\mathbf{k}$ be a field, denote by $Cat_{\mathbf{k}}$ the category of $\mathbf{k}$-linear categories and by $Cat$ the category of (ordinary) categories. Consider the forgetful functor $U:Cat_{\mathbf{k}} \to Cat$ which forgets the linear structure on the morphisms of a linear category and let $F:Cat \to Cat_{\mathbf{k}}$ be its left adjoint. A representation of a group $G$ may be defined as a functor $\beta:G \to U(\textbf{k}\text{-mod})$, where $G$ is thought of as a category with a single object, denoted by $b$, and $\mathbf{k}\text{-mod} \in Cat_{\mathbf{k}}$ is the $\mathbf{k}$-linear category of $\mathbf{k}$-vector spaces. By adjunction, we obtain a functor $\tilde{\beta}:F(G) \to \mathbf{k}\text{-mod}$. Note that $F(G)=\mathbf{k}[G]$ is the group algebra of $G$, thought of as a $\textbf{k}$-linear category with the single object $b$, and $\tilde{\beta} (b)=M$ is a left $\mathbf{k}[G]$-module. The two perspectives provided by the adjunction $(F,U)$ are useful when studying linear representations of groups. The colimit of the functor $\beta: G \to U(\textbf{k}\text{-mod})$ is the $\mathbf{k}$-module of coinvariants $\mathbf{k} \otimes_{ \mathbf{k}[G]} M$, where the right $\mathbf{k}[G]$-module structure on $\mathbf{k}$ is given by the augmentation $\mathbf{k}[G] \to \mathbf{k}$. Consider the composition of functors $i\circ \beta: G \to U(\mathbf{k}\text{-mod}) \to \text{Ch}_{\mathbf{k}}$, where $\text{Ch}_{\mathbf{k}}$ denotes the category of $\mathbf{k}$-chain complexes and $i$ is the inclusion functor. The homotopy colimit of $i \circ \beta$, with respect to the standard model structure on $\text{Ch}_{\mathbf{k}}$, is the derived coinvariants $\mathbf{k} \otimes_{\mathbf{k}[G]}^{\mathbb{L}} M$ and a model for it can be obtained by resolving $\mathbf{k}$ through the bar resolution over $\mathbf{k}[G]$. If $G=\pi_1(X,b)$ is the fundamental group of a pointed path-connected space $(X,b)$, then representations of $G$ are classical local systems over $X$. In this case, the colimit of a representation of $G=\pi_1(X,b)$ is the homology with local coefficients and the homotopy colimit may be interpreted as a chain complex, unique up to quasi-isomorphism, which calculates such homology groups. 

We refine the above constructions and results to the case of $\infty$-representations of $\infty$-groupoids, also known as $\infty$-local systems. To make sense of this, we replace $Cat$  in the above setting by the category $Set_{\Delta}$ of simplicial sets and the $\mathbf{k}$-linear category $\mathbf{k}\text{-mod}$ by the differential graded (dg) category $Ch_{\mathbf{k}}$ of $\mathbf{k}$-chain complexes.\footnote{Note the slight change in notation here. We used $\text{Ch}_{\mathbf{k}}$ for the ordinary category of chain complexes. In the remaining of the paper, we use $Ch_{\mathbf{k}}$ for the dg-category of $\mathbf{k}$-chain complexes.} The analogue of $U$ now becomes the dg nerve functor $N_{dg}: dg Cat_{\mathbf{k}} \to Set_\Delta$, where $dg Cat_{\mathbf{k}}$ is the ordinary category of dg categories. The analogue of $F$ is a functor $\Lambda: Set_{\Delta} \to dg Cat_{\mathbf{k}}$, described explicitly in section 4. For a connected Kan complex $K$, $\Lambda(K)$ is closely related to the differential graded associative algebra (dg algebra, for short) of singular chains on the based (Moore) loop space of $|K|$ as explained in \cite{RiZe17}.

We replace the fundamental group $G=\pi_1(X,b)$ in the above discussion by the fundamental $\infty$-groupoid of $X$, namely, by the Kan complex $\text{Sing}(X)$ of singular simplices in $X$. If $X$ is path-connected and $b \in X$, we use an equivalent Kan complex $\text{Sing}(X,b)$ with a single $0$-simplex. Instead of representations $\beta: G \to U(\textbf{k}\text{-mod})$ we now consider maps of simplicial sets $\beta: \text{Sing}(X,b) \to N_{dg}Ch_{\mathbf{k}}$. In this paper, we give a model for the colimit of $\beta$ using the framework of quasi-categories of \cite{Lu09}. In the context of this article, the term \textit{colimit} will always be understood in the homotopy coherent sense as introduced in section 1.2.13 of \cite{Lu09}. By adjunction, the data of a map $\beta: \text{Sing}(X,b) \to N_{dg}Ch_{\mathbf{k}}$ is equivalent to a dg functor $\tilde{\beta}: \Lambda(\text{Sing}(X,b)) \to Ch_{\mathbf{k}}$, which we can interpret as a chain complex $\tilde{\beta}(b)=M$ equipped with an action over the dg algebra $\Lambda(\text{Sing}(X,b))(b,b)$. We proved in \cite{RiZe17} that $\Lambda(\text{Sing}(X,b))(b,b)$ is isomorphic  as a dg algebra to $\Omega C$, the cobar construction of the dg coalgebra $C$ of normalized chains on $\text{Sing}(X,b)$ with Alexander-Whitney coproduct; this observation opens up the possibility of using certain algebraic tools to study $\infty$-local systems. 

Using quasi-categories as models for $\infty$-categories we prove that the $\infty$-category of $\infty$-local systems $Loc_X^{\infty}$ is equivalent to the $\infty$-derived category of dg $\Omega C$-modules. $Loc_X^{\infty}$ is defined as the quasi-category of functors $\text{Fun}(\text{Sing}(X,b), N_{dg}Ch_{\mathbf{k}})$. We give two equivalent quasi-categorical models for the $\infty$-derived category of dg $\Omega C$-modules by taking the dg nerve of two dg categories $\text{Mod}^{\infty}_{\Omega C}$ and $\text{Mod}^{\tau}_{\Omega C}$ introduced in sections 5.2 and 5.3, respectively. The first, $\text{Mod}^{\infty}_{\Omega C}$, is defined for any dg algebra $A$ using the notion of $A_{\infty}$-morphisms of $A$-modules while the second, $\text{Mod}^{\tau}_{\Omega C}$, is a simplification of the first and it may be described in terms of twisted tensor products, a notion introduced in \cite{Br59}. In section 6, we explicitly describe weak equivalences of quasi-categories $$N_{dg}\text{Mod}^{\infty}_{\Omega C} \simeq N_{dg}\text{Mod}^{\tau}_{\Omega C} \simeq Loc_{X}^{\infty}.$$ The existence of a weak equivalence $Loc_{X}^{\infty} \simeq N_{dg}\text{Mod}^{\infty}_{\Omega C}$ is a folklore result which has been used in different contexts, as seen in Remark 5.4 of \cite{BrDy16}, for example. There are more abstract ways of proving this claim, however, here we give a direct proof based on the combinatorics of simplicial sets. Another version of this equivalence also appears in \cite{Br-Ma18}. Finally, Koszul duality as discussed in \cite{Po11} implies that for any conilpotent dg coalgebra $C$ the $\infty$-derived category of dg $C$-comodules is equivalent to the $\infty$-derived category of dg $\Omega C$-modules and, henceforth, equivalent to $Loc^{\infty}_X$; a statement shown in \cite{Br-Ma18}, and in \cite{ChHoLa18} in a dual formulation.

In section 7, we use $N_{dg}\text{Mod}^{\tau}_{\Omega C}$ to compute an explicit model for the colimit of an $\infty$-local system as a twisted tensor product between the dg coalgebra $C$ of chains on the space and the dg $\Omega C$-module determined by the $\infty$-local system. More precisely, we prove the following
\begin{theorem*}
Let $(X,b)$ be a pointed path-connected space. For any $\infty$-local system $\beta: \text{Sing}(X,b) \to N_{dg}Ch_{\mathbf{k}}$ the colimit of $\beta$ is given by the twisted tensor product $(M \otimes C, \partial_{\otimes_{\tau}}),$ where $C$ is the dg coalgebra of normalized simplicial chains in $\text{Sing}(X,b)$, $M$ is the right dg $\Omega C$-module determined by $\beta$, and $\partial_{\otimes_{\tau}}$ denotes the twisted tensor product differential as introduced in \cite{Br59}.
\end{theorem*}

An immediate consequence of our results is an alternative and more conceptual proof of Brown's classical theorem on modeling the chains on the total space of a fibration in terms of chains on the base and chains on the fiber, as explained in section 8.

We go over some preliminary notions in section 2 to keep the article relatively self-contained and accessible to a broader audience. In section 3, we review the rigidification functor described in \cite{Lu09} and its cubical version introduced in \cite{RiZe17}, which are useful for understanding the dg nerve functor $N_{dg}$ and its adjoint $\Lambda$ through a different angle but not essential for understanding the results of sections 5,6, and 7. In section 4, we define $\Lambda$ and $N_{dg}$ and discuss some of their properties. 

Our main motivation for understanding the colimit of an $\infty$-local system is the possible applications to $L$-theory through the work of Ranicki-Weiss \cite{RaWe90} where the category of fractured complexes with Poincar\'e duality is described as classical local systems. It is our goal to replace the fundamental group ring in their discussion with an algebraic model for the chains on the based loop space in hope of combining the algebraic theory of homotopy types a la Sullivan, Quillen, and Mandell with the algebraic L-theory a la Ranicki and Weiss and obtain a purely algebraic characterization of manifolds. 

Another application in mind is to describe a model for the colimit of local systems of dg categories that come up in the discussions of Mirror Symmetry through generalizations of the main results of this article. In fact, similar results to those of this article should hold for $\infty$-local systems with values in more general $\infty$-categories. Such applications and generalizations will be studied in subsequent work. 

\begin{ack} We would like to thank the referee for a detailed review and constructive comments that have improved the presentation of the paper. The first author would like to acknowledge the support by the grant Fordecyt 265667. The second author would like to thank the \textit{Max Planck Institute for Mathematics} in Bonn for their support and hospitality during his visits. 
\end{ack}

\section{Preliminaries}

\subsection{Simplicial sets}

Let $\mathbf{\Delta}$ be the category whose objects are non-empty finite ordinals $\{[0], [1], [2], ... \}$ and morphisms are order preserving maps. A \textit{simplicial set} is a functor $K: \mathbf{\Delta}^{op} \to Set$ where $Set$ denotes the category of sets. We write $K_n$ for the set $K([n])$. We denote by $Set_{\Delta}$ the category having simplicial sets as objects and natural transformations as morphisms. The \textit{standard $n$-simplex} $\Delta^n$ is the simplicial set given by $\Delta^n: [m] \mapsto \text{Hom}_{\Delta}( [m],[n])$. There is a natural bijection of sets 
\begin{equation}\label{Yoneda}
K_n \cong Set_{\Delta}(\Delta^n, K).
\end{equation}

Morphisms in the category $\mathbf{\Delta}$ are generated by functions of two types: co-faces $d_i: [n] \to [n+1]$, $0 \leq i \leq n+1$, and co-degeneracies $s_j: [n] \to [n-1]$, $0 \leq j \leq n-1$. For any simplicial set $K$ and for any integer $n \geq 0$ we have face morphisms $K(d_i): K_{n+1} \to K_n$ and degeneracy morphisms $K(s_j): K_{n-1} \to K_n$ which we we call the \textit{structure maps} of $K$. The elements of $K_0$ are called the \textit{vertices} of $K$ and for any $\sigma \in K_n$ we denote by $\text{min} \sigma$ and $\text{max}\sigma$ the first and last vertices of $\sigma$, respectively (i.e. the image in $K$ of the first and last vertices of $\Delta^n \to K$ under the identification \ref{Yoneda}). We write $|\Delta^n|$ for the topological $n$-simplex $\{ (t_1,...,t_n) \in \mathbb{R}^n : 0 \leq t_1 \leq ... \leq t_n \leq 1\} \subset \mathbb{R}^n$ and $|K|$ for the topological space obtained as the geometric realization of $K$, i.e. $|K| = \text{colim }_{\Delta^n \to K} |\Delta^n|$. 

\subsection{Necklaces}
A \textit{necklace} is a simplicial set of the form $T=B_1 \vee ... \vee B_k$ where each $B_i=\Delta^{n_i}$ is a standard $n$-simplex with $n_i \geq 1$ and the wedges mean that vertices $\text{max}(B_i)$ and $\text{min}(B_{i+1})$ are identified for $i=1,...,k-1$. Each $B_i$ is called a \textit{bead} of $T$. Define the dimension of $T$ by $\text{dim}(T)=n_1 + ... +n_k-k$. Denote by $\alpha_T$ and $\omega_T$ the first and last vertices of $T$. Necklaces form a category $Nec$ in which objects are nekclaces and morphisms are maps of simplicial sets which preserve the first and last vertices. A set of generators for morphisms in $Nec$ is described in \cite{RiZe17}. A map of simplicial sets $t: T \to K$ such that $ T \in Nec$ and $t(\alpha_T)=x$ and $t(\omega_T)=y$ is called a \textit{necklace in $K$ from $x$ to $y$}; these form a category $(Nec \downarrow K)_{x,y}$ with morphisms given by maps $ T \to S \in Nec$ forming commutative triangles. 

Later on we will see that necklaces $t: T \to K \in (Nec \downarrow K)_{x,y}$ label the cubical cells of a cubical set. A cell labeled by $t: T\to K$ will have dimension $\text{dim}(T)$; this motivates the definition of the dimension of a necklace.

\subsection{Cubical sets with connections}
In sections 3, and 4 we will use a construction of \cite{RiZe17}, the cubical rigidification functor, which is based on the notion of cubical sets with connections. For this reason we recall the definition of cubical sets with connections which were originally introduced in \cite{BrHi81}. For any integer $n \geq 1$ let $\mathbf{1}^n$ be the cartesian product of $n$ copies of the category $\mathbf{1}=\{0,1\}$ with two objects and one non-identity morphism and let $\mathbf{1}^0$  be the category with one object and one morphism. Define a category $\square_c$ with objects $\{ \mathbf{1}^0, \mathbf{1}^1, \mathbf{1}^2, ... \}$ and morphisms generated by the following functors:
\\
\textit{cubical co-face functors} $\delta^{\epsilon}_{j,n}: \mathbf{1}^n \to \mathbf{1}^{n+1}$, where $j=0,1,...,n+1$, and $\epsilon \in \{0,1\}$, defined by
\begin{eqnarray*}
\delta^{\epsilon}_{j,n}(s_1,...,s_n)=(s_1,...,s_{j-1},\epsilon,s_j,...,s_n),
\end{eqnarray*}
\textit{cubical co-degeneracy functors} $\varepsilon_{j,n}: \mathbf{1}^n \to \mathbf{1}^{n-1}$, where $j=1,...,n$, defined by
\begin{eqnarray*}
\varepsilon_{j,n}(s_1,...,s_n)=(s_1,...,s_{j-1},s_{j+1},...,s_n), \text{ and }
\end{eqnarray*}
\textit{cubical co-connection functors} $\gamma_{j,n}: \mathbf{1}^n \to \mathbf{1}^{n-1}$, where $j=1,...,n-1$, $n\geq 2$, defined by
\begin{eqnarray*}
\gamma_{j,n}(s_1,...,s_n)=(s_1,...,s_{j-1},\text{max}(s_j,s_{j+1}),s_{j+2},...,s_n).
\end{eqnarray*}
\\
A \textit{cubical set with connections} is a functor $Q: \square_c^{op} \to Set$. Denote $\partial_{j}^{\epsilon}:=Q(\delta^{\epsilon}_{j,n}): Q_{n+1} \to Q_n$. Let $Set_{\square_c}$ be the cateogry whose objects are cubical sets with connections and morphisms are natural transformations. The \textit{standard n-cube with connections} $\square_c^n$ is the functor $\square_c^{op} \to Set$ represented by $\mathbf{1}^n$, namely, $\text{Hom}_{\square_c}( \text{ - }, \mathbf{1}^n): \square_c^{op} \to Set$.  The category $Set_{\square_c}$ has a (non-symmetric) monoidal structure defined by 
\begin{eqnarray*}
K \otimes L:= \underset{\sigma: \square_c^n \to K, \tau: \square_c^m \to L} {\text{colim }} \square^{n+m}_c.
\end{eqnarray*}
Let $Cat_{\square_c}$ denote the category of categories enriched over cubical sets with connections. 

Cubical sets with connections will be appear in the following context: given a simplicial set $K$ with two vertices $x$ and $y$ we will describe a natural cubical set with connections $\mathfrak{C}_{\square}(K)(x,y)$ whose cubical $n$-cells  are labeled by objects $(t: T \to K) \in (Nec \downarrow K)_{x,y}$ with $\text{dim}(T)=n$. This cubical sets with connections models the space of paths in $|K|$ from $x$ to $y$. 

\subsection{Enriched categories and weak equivalences}
In this article we will consider several enriched categories with particular notions of weak equivalences which we now recall following \cite{Lu09} and \cite{Lu11}. 

Let $\text{Ch}_{\mathbf{k}}$ denote the (ordinary) category whose objects are chain complexes over $\mathbf{k}$ which are bounded below and morphisms are chain maps. The category $\text{Ch}_{\mathbf{k}}$ has a symmetric monoidal structure given by the tensor product of complexes. Let $dgCat_{\mathbf{k}}$ be the category of categories enriched over $\text{Ch}_{\mathbf{k}}$. Objects of $dgCat_{\mathbf{k}}$ are called \textit{dg categories}. For any dg category $\mathcal{C}$ we dentote by $\mathcal{C}(x,y)$ the $\mathbf{k}$-chain complex of morphisms between $x$ and $y$. There is an ordinary category associated functorially to $\mathcal{C}$ called the \textit{homotopy category} of $\mathcal{C}$ and denoted by $ho(\mathcal{C})$. The objects of $ho(\mathcal{C})$ are the objects of $\mathcal{C}$, morphisms are given by the $0^\text{th}$ homology groups of the mapping spaces of $\mathcal{C}$, i.e.  $ho(\mathcal{C})(x,y)= H_0(\mathcal{C}(x,y))$, and composition is induced by composition in $\mathcal{C}$. A \textit{weak equivalence of dg categories} is a dg functor $F: \mathcal{C} \to \mathcal{D} \in dgCat_{\mathbf{k}}(\mathcal{C}, \mathcal{D})$ which induces an equivalence of ordinary categories $ho(\mathcal{C}) \to ho(\mathcal{D})$ and for every pair of objects $x,y \in \mathcal{C}$, the induced map $\mathcal{C}(x,y) \to \mathcal{C}(F(x), F(y))$ is a quasi-isomorphism of chain complexes. The dg category of chain complexes, denoted by $Ch_{\mathbf{k}}$ (as opposed to $\text{Ch}_{\mathbf{k}}$) has the same objects as $\text{Ch}_{\mathbf{k}}$ and the $\mathbf{k}$-chain complex $Ch_{\mathbf{k}}((C,d_C),(D, d_D))$ is the graded $\mathbf{k}$-module generated by graded maps $f:C \to D$ with differential
\[\delta(f):= f \circ d_C - (-1)^n d_D \circ f\] 
for any $(f: C \to D ) \in Ch_{\mathbf{k}}((C,d_C),(D,d_D))_n$. 

The category $dgCat_{\mathbf{k}}$ has a symmetric monoidal structure induced by taking the cartesian product of collections of objects and tensor product of chain complexes at the level of morphisms. 

Let $Cat_{\Delta}$ be the category of categories enriched over the symmetric monoidal category of simplicial sets with cartesian product. Objects of $Cat_{\Delta}$ are called \textit{simplicial categories}. To any simplicial category $\mathcal{C}$ we may associate functorially an ordinary category $ho(\mathcal{C})$ also called the \textit{homotopy category}. It is defined similarly to the homotopy category of a dg category replacing the $0^\text{th}$ homology group with the $0^\text{th}$ homotopy group of a simplicial set. A weak equivalence is also defined similarly replacing quasi-isomorphisms by (Kan) weak equivalence of simplicial sets.

Throughout the article we use the \textit{Koszul sign rule}: whenever $x$ moves past $y$, the sign change of $(-1)^{|x||y|}$ occurs.

\subsection{Kan complexes and quasi-categories}
Recall the $i$-th horn in $\Delta^n$ is the subsimplicial set $\Lambda^n_i \subset \Delta^n$ obtained from $\Delta^n$ by deleting the interior and the face opposite to the $i$-th vertex. A \textit{Kan complex} $K$ is a simplicial set with the following \textit{horn filling} condition:
for any $0\leq i \leq n$, any map $f: \Lambda^n_i \to K$ can be filled, namely, there exists a map  $\overline{f}: \Delta^n \to K$ such that $\overline{f} \circ i= f$ where $i: \Lambda^n_i \hookrightarrow \Delta^n$ is the inclusion map. 
An example of a Kan complex is the singular complex $\text{Sing}(X)$ of a topological space $X$: the simplicial set having as $n$-simplices the set of all continuous maps $|\Delta^n| \to X$ with structure maps induced by those of $\Delta^n$. 
If $b \in X$ then we can form another Kan complex $\text{Sing}(X,b) \subset \text{Sing}(X)$ whose $n$-simplices consist of all continuous maps $|\Delta^n| \to X$ sending all vertices of $|\Delta^n|$ to $b$. If $X$ is path-connected then the inclusion $\text{Sing}(X,b) \hookrightarrow \text{Sing}(X)$ is a weak equivalence of Kan complexes, i.e. it induces a homotopy equivalence of spaces $|\text{Sing}(X,b)| \simeq |\text{Sing}(X)|$.

A \textit{quasi-category} $\mathcal{C}$ is a simplicial set with the above horn filling condition being required only for $0 < i < n$. Vertices in a quasi-category are sometimes called \textit{objects}. An example of a quasi-category is the nerve of a category. Technically, the nerve of a large category is a large quasi-category: the objects might not form a set. Thus quasi-categories may be large. The nerve functor $N: Cat \to Set_{\Delta}$ has a left adjoint $ho: Set_{\Delta} \to Cat$. The \textit{homotopy category} of a quasi-category $\mathcal{C}$ is the ordinary category $ho(\mathcal{C})$. 

For any simplicial set $K$ and any quasi-category $\mathcal{C}$, the simplicial set of simplicial set morphisms (the internal hom in $Set_{\Delta}$), denoted by $\text{Fun}(K, \mathcal{C})$, is again a quasi-category. 

Given two objects $a$ and $b$ in $\mathcal{C}$ define a simplicial set $\text{Hom}^R_{\mathcal{C}}(a,b)$, called the \textit{right hom space between $a$ and $b$}, whose $n$-simplices are
\begin{equation}
\text{Hom}^R_{\mathcal{C}}(x,y)_n= \{ f: J^n \to \mathcal{C} \in Set_{\Delta}(J^n, \mathcal{C}) : f(x)=a, f(y)=b\}
\end{equation}
where $J^n$ is the simplicial set with two vertices $x$ and $y$ obtained as the quotient of $\Delta^{n+1}$ by collapsing the face $ \Delta^{ \{0,1,...,n\} }$ to a vertex $x$ and then calling $y$ the vertex opposite to the collapsed face.
If $\mathcal{C}$ is a quasi-category then $\text{Hom}^R_{\mathcal{C}}(x,y)$ is a Kan complex (Proposition 1.2.2.3 of \cite{Lu09}). A \textit{weak equivalence of quasi-categories} $\theta: \mathcal{C} \to \mathcal{D}$ is a map of simplicial sets inducing an essentially surjective functor between homotopy categories and for any $a,b \in \mathcal{C}_0$ a weak equivalence of Kan complexes $\theta: \text{Hom}^R_{\mathcal{C}}(a,b) \to \text{Hom}^R_{\mathcal{D}}(\theta(a),\theta(b))$. Quasi-categories model $\infty$-categories as studied extensively in \cite{Lu09}. 

\subsection{Algebras and coalgebras}

Let $\mathbf{k}$ be a commutative ring. A \textit{differential graded associative algebra} (dg algebra) is a graded $\mathbf{k}$-module with a differential $d: A \to A$ of degree $-1$ which squares zero together with an associative product $\cdot: A \otimes A \to A$ of degree $0$ which is compatible with the differential, i.e. $\cdot$ is a chain map when $A \otimes A$ is equipped with the tensor differential. A dg algebra is \textit{unital} if there is a map $u: \mathbf{k} \to A$ of dg algebras such that $\cdot \circ ( u \otimes id) = id= \cdot \circ (id \otimes u)$ and it is \textit{augmented} if equipped with a map of unital dg algebras $\mu: A \to \mathbf{k}$ where $\mathbf{k}$ is considered as a dg algebra concentrated at zero with zero differential. Denote $\overline{A}= \text{ker }~\mu$. 

All the dg algebras consider in the text will be assumed to be non-negatively graded. 

Dually, a \textit{differential graded coassociative coalgebra} (dg coalgebra) is a graded $\mathbf{k}$-module  $C$ with a differential $d:A \to A$ of degree $-1$ which squares zero together with a coassociative coproduct $\Delta: C \to C\otimes C$ of degree $0$ which is compatible with the differential, i.e. $\Delta$ is a chain map when $C \otimes C$ is equipped with the tensor differential. A dg coalgebra is \textit{counital} if there is a map of dg coalgebras $\epsilon: C \to \mathbf{k}$ such that $(\epsilon \otimes id) \circ \Delta= id = (id \otimes \epsilon) \circ \Delta$ and it is \textit{coaugmented} if its equipped with a map of unital dg coalgebras $\nu: \mathbf{k} \to C$. For a counital coaugmented dg coalgebra $C$ we denote $\overline{C}=\text{coker}(\nu) \cong \text{ker } (\epsilon)$, and so we have an isomorphism $C \cong \text{ker }(\epsilon) \oplus \mathbf{k} 1$. The \textit{reduced coproduct} $\Delta': \overline{C} \to \overline{C} \otimes \overline{C}$ is defined by $\Delta'(x):= \Delta(x) - x \otimes 1 - 1 \otimes x$. We call $C$ \textit{conilpotent} if for all $x \in \overline{C}$ there is some integer $n_x \geq 0$ such that $(\Delta')^{n_x}(x)=0$. If $C$ is connected, i.e. $C_0=\mathbf{k}$, then $C$ is conilpotent. The following two examples of dg coalgebras will be relevant in the following sections.
\begin{example} Associated to any simplicial set $K$ there is a natural dg coalgebra $(C^N_*(K), \partial, \Delta)$ of normalized simplicial chains with the Alexander-Whitney coproduct $\Delta: C^N_*(K) \to C^N_*(K) \otimes C^N_*(K)$. Recall $C^N_*(K)$ is defined as the free graded $\mathbf{k}$-module generated by simplices in $K$  modulo degenerate simplices, $\partial: C^N_n(K) \to C^N_{n-1}(K)$ is induced by $\partial(\sigma)= \sum_{i=0}^n (-1)^iK(d_i)(\sigma)$ and $\Delta: C^N_*(K) \to C^N_*(K) \otimes C^N_*(K)$ by
\[ \Delta(\sigma)= \bigoplus_{p+q=n} K(f_p)(\sigma) \otimes K(l_q)(\sigma),\]
where $f_p: [p] \to [p+q]$ is $f_p(i)=i$ and $l_q: [q] \to [p+q]$ is $l_q(i)=p+i$. 
\end{example}

\begin{example} Associated to any cubical set with connections $Q$ there is a natural dg coalgebra $(C^{\square}_*(Q), \partial, \Delta_{\square})$ of normalized cubical chains with Serre diagonal $\Delta_{\square}$. Let $\tilde{C}^{\square}_*(Q)$ be the chain complex such that $\tilde{C}^{\square}_n(Q)$ is the free $\mathbf{k}$-module generated by elements of $Q_n$ with differential defined on $\sigma \in Q_n$ by $\partial (\sigma):= \sum_{j=1}^n(-1)^{j}(\partial^1_{j}(\sigma) - \partial^0_{j}(\sigma))$. Let $D_nQ$ be the submodule of $\tilde{C}^{\square}_n(Q)$ which is generated by those cells in $Q_n$ which are the image of a degeneracy or of a connection map $Q_{n-1} \to Q_n$. The graded module $D_*Q$ forms a subcomplex of $\tilde{C}^{\square}_*(Q)$. Define $C^{\square}_*(Q)$ to be the quotient chain complex $\tilde{C}^{\square}_*(Q) /D_*(Q)$. The \textit{Serre diagonal} map $\Delta_{\square}: C^{\square}_*(Q) \to C^{\square}_*(Q) \otimes C^{\square}_*(Q)$ is induced by 
\begin{equation*}
\Delta_{\square}(\sigma) = \sum (-1)^{\epsilon} \partial_{j_1}^0 ... \partial^0_{j_p} (\sigma) \otimes \partial^1_{i_1} ... \partial^1_{i_q}(\sigma)
\end{equation*}
where $\sigma \in Q_n$, the sum runs through all shuffles $\{ i_1 < ... < i_q, j_1 < ... < j_p \}$ of $\{1, ..., n\}$ and $(-1)^{\epsilon}$ is the shuffle sign. The coproduct $\Delta_{\square}$ was originally introduced by Serre and the name Serre diagonal was given in \cite{KaSa05}, where the coproduct $\Delta_{\square}$ and its properties are described. 
\end{example} 

We recall the definition of the  \textit{bar construction}. For any augmented dg algebra $(A, d_A, \cdot, \mu: A \to \mathbf{k})$ define a conilpotent dg coalgebra $(BA, d_{BA}, \Delta, \nu)$ as follows. 
\begin{equation*}
BA= \mathbf{k} \oplus \left(s \overline{A}\,\right) \oplus \left(s \overline{A}\,\right)^{\otimes 2} \oplus \left(s \overline{A}\,\right)^{\otimes 3} \oplus \cdots ,
\end{equation*}
where $s$ denotes the shift by $+1$ functor and we will write monomials in $BA$ by $\{a_1 | ... | a_k\}$ where $a_i \in s\overline{A}$. Set $d_{BA}:= -d_1 + d_2$ where
\[d_{1}\{ a_{1}|...| a_{n} \}   =\sum_{i=1}^{n} (-1)^{\epsilon_{i-1}}
\{ a_{1}|...|d_{A}(a_{i}) |...| a_{n}\},\]
\[d_{2} \{  a_{1}|...| a_{n}\}= \sum_{i=1}^{n-1} (-1)^{\epsilon_{i}}
\{ a_{1}|...| a_{i} \cdot a_{i+1} |...| a_{n}\},
\]
where $\epsilon_j= |a_1| + ... |a_i| - j+1$. The coproduct $\Delta: BA \to BA \otimes BA$ is given by deconcatenation of monomials, namely
\[ \Delta( \{ a_1 | ... | a_n \} ) = \sum_{i=0}^{n} \{a_1 | ... |a_i\} \otimes \{ a_{i+1} |...| a_n\}. \] The coaugmentation $\nu: \mathbf{k} \to BA$ is the inclusion into the first direct sum term in $BA$. The bar construction defines a functor $B$ from the category of augmented dg algebras to the category of conilpotent dg coalgebras.

There is also a functor $\Omega$ from the category of conilpotent dg coalgebras to the category of augmented dg algebras called the \textit{cobar construction} defined as follows. For any conilpotent dg coalgebra $(C, d_C, \Delta, \nu: \mathbf{k} \to C)$ define an augmented dg algebra $(\Omega C, d_{\Omega C}, \cdot, \mu)$ by letting
\[
\Omega C= \mathbf{k} \oplus s^{-1}\overline{C} \oplus (s^{-1}\overline{C} \otimes s^{-1}\overline{C}) \oplus (s^{-1}\overline{C} \otimes s^{-1}\overline{C} \otimes s^{-1}\overline{C}) \oplus ...
\] 
where $s^{-1}$ is the shift by $-1$ functor and the product $\cdot$ given by concatenation of monomials. We write monomials in $\Omega C$ as $[c_1 | ... |c_n]$ where $c_i \in s^{-1} \overline{C}$. The differential $d_{\Omega C}$ is defined by extending extending $d_{\Omega C}  =- d_C + \Delta' $ as a derivation to a map $d_{\Omega C} : \Omega C \to \Omega C$. The augmentation $\mu: \Omega C \to \mathbf{k}$ is the projection to the first direct sum factor of $\Omega C$. 

For any conilpotent dg coalgebra $C$ there is a quasi-isomorphism of dg coalgebras $\rho: C \to B \Omega C$ and for any augmented dg algebra $A$ there is quasi-isomorphism of dg algebras $\pi: \Omega B A \to A$ \cite{HuMoSt74}. 

We also have a linear map $\tau: C \rightarrowtail \overline{C} \cong s^{-1} \overline{C} \hookrightarrow \Omega C$ of degree $-1$ called the \textit{universal twisting cochain}.  We use this map in the following construction which will appear in section 4.3.  For any right dg $\Omega C$-module $(M,d_M)$ define $\partial_{\otimes \tau}: M \otimes C \to M \otimes C$ by 
\begin{equation}\label{twisteddiff}
\partial_{\otimes \tau} ( m \otimes c)= d_Mm \otimes c + (-1)^{|m|}m \otimes d_Cc + \sum_{(c)}(-1)^{|m|}( m \cdot \tau(c') ) \otimes c''
\end{equation}
where $\Delta(c)=\sum_{(c)} c' \otimes c''$. It follows that $\partial_{\otimes \tau} \circ \partial_{\otimes \tau}=0$, so $(M \otimes C, \partial_{\otimes_\tau})$ is a chain complex called the \textit{twisted tensor product} of the dg $\Omega C$-module $M$ and the dg coalgebra $C$. Moreover, the map $id \otimes \Delta: M \otimes C \to M \otimes C \otimes C$ defines a right dg $C$-comodule structure on $( M \otimes C, \partial_{\otimes_\tau} )$. The twisted tensor product construction was originally introduced in \cite{Br59} to model the singular chain complex of the total space of a fibration in terms of the chains in the base and the chains in the fiber. This algebraic construction can be generalized for any pair of a dg coalgebra $C$ and dg algebra $A$ equipped with a degree $-1$ map $\tau: C \to A$ satisfying certain Maurer-Cartan equation.

\section{Simplicial and cubical rigidification of quasi-categories}

We go over the construction of the \textit{simplicial rigidification} functor $\mathfrak{C}: Set_{\Delta} \to Cat_{\Delta}$ as described in \cite{Lu09}, review how necklaces were used in \cite{DuSp11} to describe its mapping spaces, and recall the definition of the cubical rigidification functor $\mathfrak{C}_{\square_c}: Set_{\Delta} \to Cat_{\square_c}$ as introduced in \cite{RiZe17}. 

For any integers $0 \leq  i < j$ let $P_{i,j}$ be the category whose objects are subsets of the set $\{i, i+1, ..., j\}$ containing both $i$ and $j$ and morphisms are inclusions of sets.

\begin{definition} Define a functor $\mathfrak{C}: Set_{\Delta} \to Cat_{\Delta}$ by first defining a simplicial category $\mathfrak{C}(\Delta^n)$ for all $n \geq 0$ as follows. The objects of $\mathfrak{C}(\Delta^n)$ are the elements of the set $\{0, ... , n\}$ and for any two objects $i$ and $j$ with $i \leq j$, $\mathfrak{C}(\Delta^n)(i,j)$ is the simplicial set $N(P_{i,j})$, the nerve of the category $P_{i,j}$. If $j < i$, $\mathfrak{C}(\Delta^n)(i,j)= \emptyset$. Composition of morphisms in $\mathfrak{C}(\Delta^n)$ is induced by the functor $P_{j,k} \times P_{i,k} \to P_{i,k}$ induced by taking the union of two sets. The construction of the simplicial category $\mathfrak{C}(\Delta^n)$ is functorial with respect to maps $[n] \to [m]$ in $\mathbf{\Delta}$ so we may define $\mathfrak{C}: Set_{\Delta} \to Cat_{\Delta}$  by 
\begin{equation*}
\mathfrak{C}(K):= \text{colim }_{\Delta^n \to K} \mathfrak{C}(\Delta^n). 
\end{equation*}
We refer to $\mathfrak{C}: Set_{\Delta} \to Cat_{\Delta}$ as the \textit{simplicial rigidification functor}. 
\end{definition}

\begin{remark} The mapping spaces of $\mathfrak{C}$ are described explicitly in \cite{DuSp11} in terms of necklaces. More precisely, for any simplicial set $K$ and any two $x,y \in K_0$, there is an isomorphism of simplicial sets
\begin{equation}
\mathfrak{C}(K)(x,y) \cong \underset{T \to K \in (Nec \downarrow K)_{x,y}}{\text{colim }} [ \mathfrak{C}(T)(\alpha_T, \omega_T)].
\end{equation}
Moreover, for any necklace $T \in Nec$, the simplicial set $\mathfrak{C}(T)(\alpha_T, \omega_T)$ is isomorphic to a simplicial $N$-cube $(\Delta^1)^{\times N}$ where $N=\text{dim} T$. It is shown in \cite{Lu09} that if $\mathcal{C}$ is a quasi-category and $x,y \in \mathcal{C}_0$ then $\mathfrak{C}(\mathcal{C})(x,y)$ is a Kan complex homotopy equivalent to $Hom^{R}_{\mathcal{C} }(x,y)$. 
\end{remark}

We now review the construction of the cubical version of $\mathfrak{C}$ introduced in \cite{RiZe17}. In section 4 of \cite{RiZe17}, we describe a functor $C_{\square_c}: Nec \to Set_{\square_c}$ such that $C_{\square_c}(T) \cong \square_c^{\text{dim}(T)}$ for any $T \in Nec$. The connections are needed to obtain functoriality, since certain maps in the category $Nec$ are not realized by maps of (classical) cubical sets but are realized by maps of cubical sets with connections.

\begin{definition}
 For any simplicial set $K$ define $\mathfrak{C}_{\square_c}(K)$ to be the category enriched over cubical sets with connections whose objects are given by $K_0$, mapping spaces by

\begin{equation} 
\mathfrak{C}_{\square_c}(K)(x,y):= \underset{T \to K \in (Nec \downarrow K)_{x,y}}{\text{colim }} [ C_{\square_c}(T)(\alpha_T, \omega_T)],
\end{equation}
and composition induced by wedging of necklaces. This construction defines a functor $\mathfrak{C}_{\square}: Set_{\Delta} \to Cat_{\square_c}$ called the \textit{cubical rigidification functor}. 
\end{definition} 

The functor $\mathfrak{C}$ factors through $\mathfrak{C}_{\square_c}$ as explained in Proposition 5.3 in \cite{RiZe17}, which we recall below.

\begin{proposition}\label{triang} The rigidification functor $\mathfrak{C}: Set_{\Delta} \to Cat_{\Delta}$ is naturally isomorphic to the composition of functors
\begin{equation*}
Set_{\Delta} \xrightarrow{\mathfrak{C}_{\square_c}} Cat_{\square_c} \xrightarrow{\mathfrak{T}} Cat_{\Delta},
\end{equation*}
where $\mathfrak{T}$ is defined by applying the triangulation functor $\mathcal{T}: Set_{\square_c} \to Set_{\Delta}$ on mapping spaces. 
\end{proposition}

\section{The dg nerve functor and its left adjoint}

We define the dg nerve functor $N_{dg}: dgCat_{\mathbf{k}} \to Set_{\Delta}$ by first defining its left adjoint  $\Lambda: Set_{\Delta} \to dgCat_{\mathbf{k}}$ and then discuss some of its properties. 
We start with some notation. For any simplicial set $K \in Set_{\Delta}$ let $(C_*(K), \partial, \Delta)$ be the dg coalgebra of simplicial chains on $K$ over $\mathbf{k}$ with Alexander-Whitney coproduct $\Delta: C_*(K) \to C_*(K) \otimes C_*(K)$. Consider the the map $\partial'(x)= \sum_{i=1}^{n-1} (-1)^iK(d_i)$ obtained by dropping first and last terms in the definition of $\partial$.  The differential $\partial'$ and the coproduct $\Delta'$ (as defined in  section 2.6) define differential graded coassociative coalgebra structures on $C_*(K)$ and on the shifted graded module $s^{-1}C_{*>0}(K)$.

\begin{definition} Let $K$ be a simplicial set. Define $\Lambda(K)$ to be the dg category whose objects are the elements of $K_0$ and for any two $x,y \in K_0$ we define a chain complex
$(\Lambda(K)(x,y), d_{\Lambda})$ as follows. As a $\mathbf{k}$-module, $\Lambda(K)(x,y)$ is the quotient of the free $\mathbf{k}$-module generated by monomials $(\sigma_1 | ... | \sigma_k)$, where each $\sigma_i$ is a generator of  $s^{-1}C_{*>0}( K)$ satisfying $\text{max} \sigma_i = \text{min} \sigma_{i+1}$, by the equivalence relation generated by
\begin{equation*}
 (\sigma_1 |... | \sigma_k )\sim (\sigma_1 | ... |\sigma_{i-1}  | \sigma_{i+1} | ... |\sigma_k)
 \end{equation*}
if $k\geq 2$ and $\sigma_i$ is a degenerate $1$-simplex for some $1 \leq i \leq k$; and
 \begin{equation*}
 (\sigma_1 | ... | \sigma_k )\sim 0
 \end{equation*}
 if $k \geq 1$ and $\sigma_i \in C_{n_i}(K)$ is a degenerate simplex for some $1 \leq i \leq k$ and $n_i \geq 2$. Denote by $[\sigma_1 | ... | \sigma_k]$ the equivalence class of $(\sigma_1 | ... | \sigma_k)$. Composition is given by concatenation of monomials, namely, 
 $[\sigma_1 | ... | \sigma_k] \in \Lambda(K)(x,y)$ and $[\gamma_1 | ... | \gamma_k] \in \Lambda(K)(y,z)$ then $[\gamma_1 | ... | \gamma_k]  \circ [\sigma_1 | ... | \sigma_k] = [\sigma_1 | ... | \sigma_k | \gamma_1 | ... | \gamma_k] \in \Lambda(K)(x,z)$. The differential $d_{\Lambda}: \Lambda(K)(x,y) \to \Lambda(K)(x,y) $ is defined by extending  $d_{\Lambda}= -\partial' + \Delta'$ as a derivation on monomials, similarly to how the differential of the cobar construction is defined. The map $d_{\Lambda}$ is well defined on equivalence classes and $d_{\Lambda} \circ d_{\Lambda}=0$. This construction is clearly functorial on $K$ and thus defines a functor $$\Lambda: Set_{\Delta} \to dgCat_{\mathbf{k}}.$$
 \end{definition}
 
 \begin{remark} Note that a monomial $(\sigma_1 | ... | \sigma_k)$ where $\sigma_i \in K_{n_i}$ and $\text{min} \sigma_1=x$, $\text{max} \sigma_i = \text{min} \sigma_{i+1}$ for $i=1,...,k-1$, and $\text{max}\sigma_k=y$ corresponds to a necklace $t: \Delta^{n_1} \vee ... \vee \Delta^{n_k} \to (Nec \downarrow K)_{x,y}$. 
\end{remark} 

\begin{definition}
The \textit{dg nerve} functor $N_{dg}: dgCat_{\mathbf{k}} \to Set_{\Delta}$ is defined by setting
\begin{equation*}
N_{dg}(\mathcal{C})_n:= Hom_{dgCat_{\mathbf{k}}}(\Lambda(\Delta^n), \mathcal{C})
\end{equation*}
for any dg category $\mathcal{C}$. 
\end{definition}

\begin{remark} It follows from Theorem 6.1 of \cite{RiZe17} that the above definition of $N_{dg}$ agrees with Lurie's definition of the dg nerve functor in \cite{Lu11}. Hence $\Lambda: Set_{\Delta} \to dgCat_{\mathbf{k}}$ is the left adjoint for Lurie's dg nerve functor.  Moreover, it is shown in \cite{Lu11} that the simplicial set $N_{dg}(\mathcal{C})$ is a quasi-category.
\end{remark}

It follows from \cite{RiZe17} that $\Lambda$ is weakly equivalent to taking chains on the mapping spaces of $\mathfrak{C}$, as we recall next.

\begin{proposition} \label{LandC} For any $K \in Set_{\Delta}$ there is a natural weak equivalence of dg categories $\Lambda(K) \simeq (\mathfrak{Q} \circ \mathfrak{C}) (K)$ where $\mathfrak{C}: Set_{\Delta} \to Cat_{\Delta}$ is the rigidification functor and $\mathfrak{Q}: Cat_{\Delta} \to dgCat_{\mathbf{k}}$ is the functor which applies the simplicial chains functor on the mapping spaces of a simplicial category. 
\end{proposition}

\begin{proof} For any two $x,y \in K$, the chain complex $\Lambda(K)(x,y)$ is isomorphic to the chain complex of normalized cubical chains on $\mathfrak{C}_{\square_c}(K)(x,y)$, as explained in section 6 of \cite{RiZe17}. Then by Proposition \ref{triang}, there is an isomorphism of simplicial sets $\mathfrak{C}(K)(x,y) \cong \mathcal{T}(\mathfrak{C}_{\square_c}(K)(x,y) )$ where $\mathcal{T}: Set_{\square_c} \to Set_{\Delta}$ is the triangulation functor. Finally, the desired result follows from Lemma 7.2 of \cite{RiZe17} which says that there is a natural quasi-isomorphism between the chain complex of normalized simplicial chains on $\mathcal{T}(\mathfrak{C}_{\square_c}(K) )$ and the chain complex of normalized cubical chains on $\mathfrak{C}_{\square_c}(K)$. 
\end{proof}

By Proposition 6.2 of \cite{DuSp11} there is a natural weak equivalence of simplicial categories $\mathfrak{C}(K \times L) \simeq \mathfrak{C}(K) \times \mathfrak{C}(L)$. Hence, by Proposition \ref{LandC} and a classical result of Eilenberg-Zilber it follows that there is a weak equivalence of dg categories $\Lambda(K \times L) \simeq \Lambda(K) \otimes \Lambda(L)$. We describe explicitly the map $AW_{\Lambda}: \Lambda(K \times L ) \to \Lambda(K ) \otimes \Lambda(L)$. 

\begin{proposition} \label{oplax} For any simplicial sets $K$ and $L$, there exists a dg functor $AW_{\Lambda}: \Lambda(K \times L ) \to \Lambda(K ) \otimes \Lambda(L)$, endowing $\Lambda: Set_{\Delta} \to dgCat_{\mathbf{k}}$ with an oplax monoidal structure.
\end{proposition}
\begin{proof}
Denote by $\mathfrak{Q}_{\square_c}: Cat_{\square_c} \to dgCat_{\mathbf{k}}$ the functor defined by applying normalized cubical chains at the level of mapping spaces. From the construction of $\Lambda$ it  follows that there is a natural isomorphism of functors $\Lambda \cong \mathfrak{Q}_{\square_c} \circ \mathfrak{C}_{\square_c}$. Define $AW_{\Lambda}: \Lambda(K \times L ) \to \Lambda(K ) \otimes \Lambda(L)$ as follows.  Define a dg functor 
\begin{equation*}
\Delta_{\square}: \mathfrak{Q}_{\square_c}( \mathfrak{C}_{\square_c} (K \times L) ) \to \mathfrak{Q}_{\square_c}( \mathfrak{C}_{\square_c} (K \times L) ) \otimes \mathfrak{Q}_{\square_c}( \mathfrak{C}_{\square_c} (K \times L) )
\end{equation*}
to be the diagonal map on objects and the Serre diagonal map applied to the cubical sets with connections on mapping spaces (as recalled in Example 2.2). Consider the dg functor $$P_1 \otimes P_2:  \mathfrak{Q}_{\square_c}( \mathfrak{C}_{\square_c} (K \times L) ) \otimes \mathfrak{Q}_{\square_c}( \mathfrak{C}_{\square_c} (K \times L) ) \to \mathfrak{Q}_{\square_c} (\mathfrak{C}_{\square_c}(K) )\otimes \mathfrak{Q}_{\square_c} (\mathfrak{C}_{\square_c}(L)),$$ where $P_1: \mathfrak{Q}_{\square_c}( \mathfrak{C}_{\square_c} (K \times L) ) \to \mathfrak{Q}_{\square_c}( \mathfrak{C}_{\square_c} (K))$ is induced by the projection $K \times L \to K$ and $P_2$ is defined similarly. Finally, define

\begin{equation} 
\mathfrak{Q}_{\square_c}( \mathfrak{C}_{\square_c} (K \times L) ) \to \mathfrak{Q}_{\square_c} (\mathfrak{C}_{\square_c}(K) )\otimes \mathfrak{Q}_{\square_c} (\mathfrak{C}_{\square_c}(L))
\end{equation}
to be the composition  $ P_1 \otimes P_2 \circ \Delta_{\square}$.  Using the isomorphism $\Lambda \cong \mathfrak{Q}_{\square_c} \circ \mathfrak{C}_{\square_c}$, the above functor yields a functor
$$AW_{\Lambda}: \Lambda(K \times L ) \to \Lambda(K ) \otimes \Lambda(L).$$ This construction endows $\Lambda$ with an oplax monoidal structure given by $AW_{\Lambda}.$
\end{proof}

Recall the following observation from  \cite{RiZe17} (Theorem 7.1) which opens up the possibility of using algebraic techniques to study $\infty$-local systems as we shall see throughout this article.
\begin{theorem} \label{cobar} Let $K$ be a simplicial set such that $K_0=\{x\}$. Then there is an isomorphism of dg algebras $\varphi: \Omega (C^N_*(K), \partial, \Delta) \xrightarrow{\cong} \Lambda(K)(x,x)$. 
\end{theorem}
The isomorphism of dg algebras $\varphi$ is determined by $\varphi[\sigma]=[\sigma]$ if $\text{deg}( \sigma) >1$ and $\varphi[\sigma]= [\sigma]- c_x$ if $\text{deg}(\sigma)=1$ where $c_x \in \Lambda(K)(x,x)$ corresponds to the single bead necklace $\Delta^1 \to K$ which is degenerate at $x$. 
\begin{remark} By Proposition \ref{oplax} and  Theorem \ref{cobar} it follows that, for any $K$ with $K_0=\{x\}$, $AW_{\Lambda}$ induces a coproduct on $\Omega (C^N_*(K), \partial, \Delta)$ which is compatible with concatenation of monomials giving rise to a dg bialgebra structure on $\Omega (C^N_*(K), \partial, \Delta)$. This algebraic structure was already discussed by Baues in \cite{Ba98}; we recall the explicit formula for the coproduct. Given a subset $a=\{a_0 < a_1  < ... < a_r\} \subset \{0, ... , n\}$ we obtain an injective function $i_a: \{0 ,..., r \} \to \{0,...,n\}$ with the subset $a$ as image. For any simplex $\sigma \in K_n$ define $\sigma(a)= i^*_a(\sigma) \in K_r$. Let $n\geq 2$ and denote by $x$ and $y$ the first and last vertices of $\sigma \in K_n$. Then the coproduct $AW_{\Lambda}: \Lambda(K)(x,y) \to \Lambda(K)(x,y) \otimes \Lambda(K)(x,y)$ is given by
\begin{equation*}
AW_{\Lambda} [\sigma] =\sum_{ a \leq \{0,...,n\} } (-1)^{\epsilon(a)}[ \sigma(a_0, a_0+1, ..., a_1), ..., \sigma(a_{r-1}, a_{r-1}+1,...,a_r)] \otimes [\sigma(a)],
\end{equation*}
where the sum runs through all subsets $a=(a_0,...,a_r) \subset \{0,...,n\}$ with $a_0=0 < a_1 < ... < a_r=n$ and $\epsilon(a)= \sum_{i=1}^r(i-1)( \text{dim} (\sigma(a_{i-1},...,a_i))  -1).$ If $n=1$, so $\text{deg}[\sigma]=0$, then $AW_{\Lambda} [\sigma]= [\sigma] \otimes [\sigma]$. For monomials of arbitrary length this formula is extended as an algebra map. Baues discusses how such formula may be obtained from the Serre diagonal and a cubical interpretation of the cobar construction. Hence, we can think of $AW_{\Lambda}$ as an extension of Baues' coproduct on the cobar construction to a functor of dg categories.
\end{remark}

Finally, we recall one more result from \cite{RiZe17} (Corollary 7.9) which follows from Theorem \ref{cobar} together with some basic results from the theory of quasi-categories. It extends a classical theorem of Adams, the main result of \cite{Ad52}, to path-connected spaces which are not necessarily simply connected:

\begin{corollary}\label{rize17} Let $(X,b)$ be a path-connected pointed space and $C= (C^N_*(X,b), \partial, \Delta)$ the connected dg coalgebra of normalized singular chains with vertices at $b \in X$ and Alexander-Whitney coproduct. Then the cobar construction $\Omega C$  is weakly equivalent as a dg algebra to  $C_*(\Omega_bX)$, the singular chains on the space of (Moore) loops in $X$ based at $b$ with product induced by concatenation of loops.
\end{corollary}

\section{Models for the $\infty$-category of $\infty$-local systems}
We introduce three quasi-categories associated to a path-connected pointed space $(X,b)$. From now on we assume $\mathbf{k}$ is a field.

\subsection{As the quasi-category of functors} Let $\text{Sing}(X,b)\subset \text{Sing}(X)$  be the sub Kan complex whose set of $n$-simplices consists of all continuous maps $|\Delta^n| \to X$ which map the vertices of $|\Delta^n|$ to $b$. 

\begin{definition} An $\infty$-\textit{local system} of $\mathbf{k}$-chain complexes over $X$ is a morphism of simplicial sets $F: \text{Sing}(X,b) \to N_{dg} Ch_{\mathbf{k}}$. 
We denote by $Loc^{\infty}_X$  the quasi-category of $\infty$-local systems of $\mathbf{k}$-chain complexes over $X$, namely 
\begin{eqnarray}
Loc^{\infty}_X := \text{Fun}( \text{Sing}(X,b), N_{dg}Ch_{\mathbf{k}}).
\end{eqnarray}
\end{definition}

The $n$-simplices of $Loc^{\infty}_X$ are given by
\begin{eqnarray}
(Loc^{\infty}_X)_n= Set_{\Delta} (\Delta^n \times \text{Sing}(X,b), N_{dg}Ch_{\mathbf{k}})
\end{eqnarray}
By adjunction, we have a natural isomorphism
\begin{eqnarray} \label{nsimpinfty}
(Loc^{\infty}_X)_n \cong dgCat_{\mathbf{k}} (\Lambda(\Delta^n \times \text{Sing}(X,b)), Ch_{\mathbf{k}}).
\end{eqnarray}
The simplicial set $Loc^{\infty}_X$ is a quasi-category because $N_{dg}Ch_{\mathbf{k}}$ is a quasi-category. The notion of an $\infty$-local system was defined in \cite{BlSm14} as a refinement of the classical notion of a local system. 

\subsection{As the dg category of dg modules with morphisms between bar constructions}

The data of a local system $F: \text{Sing}(X,b) \to N_{dg}Ch_{\mathbf{k}}$ determines a morphism of dg categories $\tilde{F}: \Lambda(\text{Sing}(X,b) ) \to Ch_{\mathbf{k}}$. Note $\Lambda(\text{Sing}(X,b) ) $ has a single object $b$ so $\tilde{F}$ determines a chain complex $\tilde{F}(b)$ together with a right dg module structure over the dg algebra $\Lambda(\text{Sing}(X,b) ) (b,b) \cong \Omega C$, where $C=(C^N_*(\text{Sing}(X,b); \mathbf{k}), \partial, \Delta)$ is the connected dg coalgebra of normalized chains on $\text{Sing}(X,b)$ with Alexander-Whitney coproduct $\Delta: C \to C\otimes C$. 
\\

We recall certain categorical constructions associated to any dg algebra $(A,d_A)$.

\begin{definition} Let $\text{Mod}_A$ be the dg category having as objects right dg $A$-modules which are bounded below and for any two such objects $(M, d_M)$ and $(N, d_N)$ a chain complex $(\text{Mod}_A(M,N), \delta)$ of morphisms defined as follows. As a graded $\mathbf{k}$-module  $\text{Mod}_A(M,N) = \bigoplus_{p \in \mathbb{Z} } \text{Mod}_A(M,N)_p$, where $\text{Mod}_A(M,N)_p$ is the $\mathbf{k}$-module of degree $p$ maps $f: M \to N$ of right $A$-modules. The differential $\delta: \text{Mod}_A(M,N)_* \to \text{Mod}_A(M,N)_{*-1}$ is defined by 
\begin{equation*}
\delta (f)=  d_N \circ f - (-1)^{|f|} f \circ d_M.
\end{equation*}
Composition of morphisms is defined in the obvious way. The \textit{derived} $\infty$-\textit{category} of right dg $A$-modules is the quasi-category $N_{dg}((\text{Mod}_A)_{\text{proj}})$, where $(\text{Mod}_A)_{\text{proj}}$ is the full sub dg category of $\text{Mod}_A$ consisting of projective right dg $A$-modules. 
\end{definition}
The homotopy category of the derived $\infty$-category of right dg $A$-modules has projective modules as objects and morphisms given by chain homotopy classes of chain maps. Thus the homotopy category is isomorphic to the classical derived category of right dg $A$-modules since quasi-isomorphisms between projective modules are invertible up to chain homotopy \cite{Lu11}.

We can also construct a dg category of dg right $A$-modules whose homotopy category consists of $A_{\infty}$-module morphisms. This dg category provides another model for the derived $\infty$-derived category which we will be useful for us. We now assume $(A,d_A)$ is augmented. 

\begin{definition} Let $\text{Mod}^{\infty}_{A}$ be the dg category having the same objects as $\text{Mod}_A$ and for any two such objects $(M, d_M)$ and $(N, d_N)$ a chain complex $(\text{Mod}^{\infty}_{A}(M,N), \delta)$ of morphisms defined as follows. 
As a graded $\mathbf{k}$-module  $\text{Mod}^{\infty}_A(M,N) = \bigoplus_{p \in \mathbb{Z} } \text{Mod}^{\infty}_A(M,N)_p$, where $\text{Mod}^{\infty}_A(M,N)_p$ is the $\mathbf{k}$-module generated by degree $p$ $BA$-comodule maps 
\begin{equation*}
f: M \otimes BA \to N \otimes BA,
\end{equation*}
where $BA$ is the bar construction of $A$ as defined in section 2.6.  The differential $\delta: \text{Mod}^{\infty}_A(M,N)_* \to \text{Mod}^{\infty}_A(M,N)_{*-1}$ is defined by
\begin{equation*}
\delta (f)=  b_N \circ f - (-1)^{|f|} f \circ b_M,
\end{equation*}
where for any right dg $A$-module $M$,  $b_M: M \otimes BA \to M \otimes BA $ is given by
\begin{equation}
\begin{split}
&b_M(m \otimes \{ a_1 | ... | a_n \} ) := d_Mm \otimes \{a_1 | ... |a_n\} - \sum_{i=1}^n m \otimes (-1)^{\epsilon_i} \{ a_1| ...|d_Aa_i| ... |a_n\}\\
&+ (-1)^{|m|} (m \cdot a_1) \otimes \{ a_2 | ... | a_n \} + \sum_{i=1}^{n-1} (-1)^{\epsilon_i} m \otimes \{ a_1 | ... |a_i a_{i+1} |... | a_n \}
\end{split}
\end{equation}
for any $m \in M$, $\{a_1 | ... |a_n\} \in BA$ and $\epsilon_i= |m| + |a_1| + ... +|a_i| -i +1$. Composition of morphisms is defined in the obvious way.
\end{definition}

\begin{remark} Since $N \otimes BA$ is a free right $BA$-comodule, the chain complex $\text{Mod}^{\infty}_{A}(M,N)$ is equivalent to the vector space generated by linear maps
\begin{equation*}
f: M \otimes BA \to N
\end{equation*}
together with differential $\hat{\delta}$ defined by 
\begin{equation}\label{bar}
\begin{split}
&\hat{\delta}(f)(m \otimes \{ a_1 | ... |a_n \} ) := d_Nf ( m \otimes \{ a_1 | ... |a_n \} ) - (-1)^{|f|} f ( b_M ( m \otimes \{ a_1 | ... |a_n \} ))\\
& + (-1)^{|m|+ |a_1| + ... + |a_{n-1}| - n+1} f( m \otimes \{ a_1 | ... | a_{n-1} \} ) \cdot a_n.
\end{split}
\end{equation}
The $0$-cycles in $\text{Mod}^{\infty}_{A}(M,N)$ are, by definition, $A_{\infty}$-module morphisms between the dg $A$-modules $M$ and $N$ and $H_0(\text{Mod}^{\infty}_{A}(M,N))$ is the vector space of $A_{\infty}$-module morphisms module chain homotopy. Recall that an $A_{\infty}$-module quasi-isomorpshism is invertible up to chain homotopy \cite{Ke01}.
\end{remark}

\begin{proposition} There is a weak equivalence of quasi-categories $$N_{dg}(\text{Mod}^{\infty}_{A}) \simeq N_{dg}((\text{Mod}_A)_{\text{proj}}).$$
\end{proposition}
\begin{proof} We first show that $(\text{Mod}_A)_{\text{proj}}$ and $ \text{Mod}^{\infty}_{A}$ are weak equivalent as dg categories. Consider the dg functor $\iota: (\text{Mod}_A)_{\text{proj}} \to \text{Mod}^{\infty}_{A}$ which is identity on objects and for any two projective modules $P$ and $Q$ the chain map $\iota: (\text{Mod}_A)_{\text{proj}}(P,Q) \to \text{Mod}^{\infty}_{A}(P,Q)$ is given on any $f: P \to Q$ by defining $\iota(f): P \otimes BA \to Q \otimes BA$ to be
$$\iota(f) (p \otimes 1_{\mathbf{k}} )= f(p)$$
and 
$$\iota(f) (p \otimes \{a_1 | ... | a_n \})=0$$ if $n>0$. The dg functor $\iota$ induces an essentially surjective functor of homotopy categories $H_0(\iota): H_0((\text{Mod}_A)_{\text{proj}}) \to H_0(\text{Mod}^{\infty}_{A})$. In fact, given a dg $A$ module $M \in H_0(\text{Mod}^{\infty}_{A})$ we may consider its bar resolution $B(M,A,A)= M \otimes BA \otimes A \in H_0((\text{Mod}_A)_{\text{proj}})$, which is quasi-isomorphic to $M$ as a right dg $A$-module. Since $B(M,A,A)$ is $A_{\infty}$-quasi-isomorphic to $M$ it follows that $B(M, A,A)$ are isomorphic in $H_0(\text{Mod}^{\infty}_{A})$.

The induced map $H_0(\iota): H_0((\text{Mod}_A)_{\text{proj}}(P,Q)) \to H_0(\text{Mod}^{\infty}_{A}(P,Q))$ is injective since a morphism of dg $A$-modules determines a unique $A_{\infty}$-morphism and a chain homotopy of maps of dg $A$-modules determines a unique chain homotopy between the corresponding $A_{\infty}$-morphisms. Moreover, $H_0(\iota): H_0((\text{Mod}_A)_{\text{proj}}(P,Q)) \to H_0(\text{Mod}^{\infty}_{A}(P,Q))$ is surjective since a morphism in $H_0(\text{Mod}^{\infty}_{A}(P,Q))$  represented by a map of  dg $BA$-comodules $f: P \otimes BA \to Q \otimes BA$ induces a morphism between bar resolutions $f \otimes \text{id}_{A}: B(P,A,A) \to B(Q,A,A)$, which provides a lift of $f$ to $H_0((\text{Mod}_A)_{\text{proj}}(P,Q))$. Since, for any $n\geq 0$, we have $H_n((\text{Mod}_A)_{\text{proj}}(P,Q))= H_0((\text{Mod}_A)_{\text{proj}}(P,Q[n]))$, where $Q[n]$ denotes the graded module shifted by $n$, and similarly for $H_n (\text{Mod}^{\infty}_{A}(P,Q))$, it follows that $H_n(\iota)$ is an isomorphism for all $n$. 

Finally, since the the dg nerve functor $N_{dg}$ sends weak equivalences of dg categories to weak equivalences of quasi-categories (as explained in Proposition 1.3.1.20 of \cite{Lu11}), it follows that $N_{dg}(\iota): N_{dg}(\text{Mod}^{\infty}_{A}) \to N_{dg}((\text{Mod}_A)_{\text{proj}})$ is a weak equivalence of quasi-categories. 
\end{proof}

\subsection{As the dg category of dg modules with morphisms between twisted tensor products}
If $A= \Omega C$ for a conilpotent dg coalgebra $C$ we define a new dg category $\text{Mod}^{\tau}_{\Omega C}$  of dg $\Omega C$-modules by replacing the bar construction in the above discussion with a twisted tensor product construction.

\begin{definition} Define a dg category $\text{Mod}^{\tau}_{\Omega C}$  as follows. The objects of $\text{Mod}^{\tau}_{\Omega C}$ are the same objects as in $\text{Mod}^{\infty}_{\Omega C}$, i.e. right dg $\Omega C$-modules. Given two dg $\Omega C$-modules $(M, d_M)$ and $(N, d_N)$ the chain complex $\text{Mod}^{\tau}_{\Omega C}(M,N)$ is defined as the vector space generated by graded right $C$-comodule maps 
\begin{eqnarray*}
f: M \otimes C \to N \otimes  C,
\end{eqnarray*}
with differential $\delta: \text{Mod}^{\tau}_{\Omega C}(M,N) \to \text{Mod}^{\tau}_{\Omega C}(M,N)$ defined by
\begin{eqnarray*}
\delta(f)= \partial_{\otimes_{\tau}} \circ f - (-1)^{|f|} f \circ \partial_{\otimes_{\tau}},
\end{eqnarray*}
where $\partial_{\otimes_{\tau}}$ denotes the differential of the twisted tensor product construction associated to the universal twisting cochain $\tau: C \to \Omega C$, as defined by equation \ref{twisteddiff}. Composition of morphisms is defined in the obvious way.
\end{definition}

Since $N \otimes C$ is a free right $C$-comodule, $\text{Mod}^{\tau}_{\Omega C}(M,N)$ is isomorphic to the chain complex generated by graded linear maps 
\begin{eqnarray*}
g: M \otimes C \to N
\end{eqnarray*}
with differential $\hat{\delta}$ defined by
\begin{equation} \label{deltahat}
\begin{split}
&\hat{\delta} (g)(m \otimes \sigma):= d_N g(m \otimes \sigma) -  (-1)^{|g|}g(d_M m \otimes \sigma) - (-1)^{|m| +|g|} g( m \otimes \partial(\sigma)) \\
&+ \sum_{(\sigma)} (-1)^{|m| +|g|} g ( (m \cdot \tau(\sigma') )\otimes \sigma'') + \sum_{(\sigma)} (-1)^{|g|+|m| + |\sigma'|} g(m \otimes \sigma') \cdot \tau(\sigma''),
\end{split}
\end{equation}
where we have written $\Delta(\sigma)=\sum_{(\sigma)} \sigma' \otimes \sigma''$ for the coproduct $\Delta: C \to C \otimes C$.

\section{Equivalence of $\infty$-categories $Loc_X^{\infty}$, $N_{dg} \text{Mod}^{\infty}_{\Omega C}$, and $N_{dg} \text{Mod}^{\tau}_{\Omega C}$}

In the previous sections we have introduced three quasi-categories: $Loc_X^{\infty}$, $N_{dg} \text{Mod}^{\infty}_{\Omega C}$, and $N_{dg} \text{Mod}^{\tau}_{\Omega C}$. In this section we argue that all of these are weakly equivalent as quasi-categories, i.e. equivalent as $\infty$-categories. We do this in the following two subsections

\subsection{Equivalence between $N_{dg}  \text{Mod}^{\infty}_{\Omega C} $ and $N_{dg} \text{Mod}^{\tau}_{\Omega C}$}\
\\

We first show that the two dg categories $\text{Mod}^{\tau}_{\Omega C}$ and $\text{Mod}^{\infty}_{\Omega C}$ are weakly equivalent.

\begin{proposition}  \label{bar complex} For any conilpotent dg coalgebra $C$ and any right dg $\Omega C$-module $M$ there is a natural quasi-isomorphism of chain complexes 
\begin{equation}
\phi:  (M \otimes B\Omega C, b_M) \to (M \otimes C, \partial_{\otimes_{\tau}})  .
\end{equation}
\end{proposition}
\begin{proof}
Define $\phi:  (M \otimes B\Omega C, b_M) \to (M \otimes C, \partial_{\otimes_{\tau}})$ by setting
$\phi(m \otimes \{ a_1 | ... |a_n\} ) =0$ if $n>1$, and if $n=1$ with $a_1=[c_1| ... |c_k]$ let
\[
\phi(m \otimes \{ [c_1 | ... |c_k] \} )=\left\{
  \begin{array}{lll}
    c_1,                                              &  k=1,\\
  m \cdot [c_1 | ... |c_{k-1} ] \otimes c_k, & k>1;
  \end{array}
\right.
\]
It is an easy computation to check that $\phi$ is a chain map. Moreover, $\phi$ is surjective with right inverse given by the chain map $id \otimes \rho_C: (M \otimes C, \partial_{\otimes_{\tau}}) \to  (M \otimes B\Omega C, b_M)$ where $\rho_C: C \to B\Omega C$ is the dg coalgebra map defined by
\begin{equation}
\rho_C(c)= \{ [c ] \} + \sum_{(c)} \{ [c' ] | [c''] \} + \sum_{(c)} \{ [c' ] | [c'' ] | [c'''] \} + ... \text{  , }
\end{equation}
and the number of prime subscripts denotes the number of iterated applications of $\Delta: C \to C \otimes C$; namely, $((\text{id}_C \otimes \Delta) \circ \Delta ) (c) = \sum_{(c)} c' \otimes c'' \otimes c'''$, $((\text{id}_C \otimes \text{id}_C \otimes \Delta) \circ (\text{id}_C \otimes \Delta) \circ \Delta)(c) = \sum_{(c)} c' \otimes c'' \otimes c''' \otimes c''''$, and so on. Note that $\rho_C$ is well defined by the conilpotency of the coalgebra $C$.

We argue that $(\text{ker } \phi, b_M)$ is a contractible sub complex in order to conclude that $\phi$ is a quasi-isomorphism. In fact, define $h: \text{ker } \phi \to \text{ker } \phi$ on any $m  \otimes \{a_1 | ... |a_n \} \in \text{ker } \phi$ with $a_1=[c_1 | ... |c_k] \in \Omega C$ by 

\[
h(m \otimes \{ [c_1| ... |c_k] | a_2 |  ... | a_n \} )=\left\{
  \begin{array}{lll}
    0,                                              &  k=1,\\
 \{ [c_1 ] | [c_2 | ... |c_k] | a_2 | ... |a_n \} & k>1;
  \end{array}
\right.
\]
The conilpotency of $C$ yields that for any $x \in \text{ker } \phi$ there exists a non-negative integer $n_x$ such that $(b_M \circ h + h \circ b_M - id)^{n_x}=0$. This last equation implies that if $x \in \text{ker } \phi$ is a cycle then there exists some $y$ such that $x=b_M(y)$, as desired. 
\end{proof}

\begin{corollary}\label{equiv1}
For any conilpotent dg coalgebra $C$ the dg categories $\text{Mod}^{\tau}_{\Omega C}$ and $\text{Mod}^{\infty}_{\Omega C}$ are weakly equivalent.
\end{corollary}
\begin{proof}
Recall that the objects of the two dg categories $\text{Mod}^{\tau}_{\Omega C}$ and $\text{Mod}^{\infty}_{\Omega C}$ are the same: right dg $\Omega C$-modules. Consider the functor $F: \text{Mod}^{\infty}_{\Omega C} \to \text{Mod}^{\tau}_{\Omega C}$ which is identity on objects and on morphisms the map $F: \text{Mod}^{\infty}_{\Omega C}(M,N) \to \text{Mod}^{\tau}_{\Omega C}(M,N)$
is defined by sending a a morphism $(f: M \otimes B\Omega C \to N) \in \text{Mod}^{\infty}_{\Omega C}(M,N) $ to the composition
\begin{equation}
F(f): M \otimes C \xrightarrow{ id \otimes \rho_C} M \otimes B\Omega C \xrightarrow{f} N. 
\end{equation}
The map $F: \text{Mod}^{\infty}_{\Omega C}(M,N) \to \text{Mod}^{\tau}_{\Omega C}(M,N)$ is a chain map. Moreover, $F$ is surjective and a right inverse is given by the chain map $G: \text{Mod}^{\tau}_{\Omega C}(M,N) \to \text{Mod}^{\infty}_{\Omega C}(M,N)$ which sends a morphism $(g: M \otimes C \to N) \in \text{Mod}^{\tau}_{\Omega C}(M,N) $ to the map $G(g): M \otimes B \Omega C \to N$ defined by $G(g)(m \otimes \{ a_1 | ... | a_n \})=0$ if $n>1$ and if $n=1$ with $a_1=[c_1|...|c_k]$ 
\begin{equation}
G(g)(m \otimes \{ [c_1|...|c_k] \} ) := \sum_{i=0}^{k} g(m \cdot [c_1|...|c_i] \otimes c_{i+1} ) \cdot [c_{i+2} | ... | c_k].
\end{equation}
We argue that the kernel of $F: \text{Mod}^{\infty}_{\Omega C}(M,N) \to \text{Mod}^{\tau}_{\Omega C}(M,N)$ is contractible. Note that $\text{ker } F$ is isomorphic to the complex generated by linear maps 
\begin{equation*}
f: \text{coker} (id \otimes \rho_C ) \to N 
\end{equation*}
with differential induced by $\hat{\delta}$ as defined in \ref{bar}. By Proposition \ref{bar complex} (and its proof) there is an isomorphism $\text{coker} (id \otimes \rho_C ) \cong \text{ker } \phi$, so we have an isomorphism of complexes
\begin{equation} \label{iso}
\text{ker } F \cong(  \mathbf{k}\{f: \text{ker } \phi \to N \} , \hat{\delta} ),
\end{equation}
where $\mathbf{k}\{f: \text{ker } \phi \to N \}$ denotes the $\mathbf{k}$-vector space generated by linear maps $f: \text{ker } \phi \to N$. Define $\tilde{h}: (  \mathbf{k}\{f: \text{ker } \phi \to N \} , \hat{\delta} ) \to (  \mathbf{k}\{f: \text{ker } \phi \to N \} , \hat{\delta} )$ on any generator $f$ by $\tilde{h}(f)= f \circ h$ where $h: \text{ker } \phi \to \text{ker } \phi$ is defined in the proof of Proposition \ref{bar complex}. For all  $f \in(  \mathbf{k}\{f: \text{ker } \phi \to N \} , \hat{\delta} )$ there exists a non-negative integer $n_f$ such that $(\tilde{h} \circ \hat{\delta} + \hat{\delta} \circ \tilde{h} - id )^{n_f}=0$. As in the proof of Proposition \ref{bar complex}, this implies that $(  \mathbf{k}\{f: \text{ker } \phi \to N \} , \hat{\delta} )$ is contractible. By \ref{iso}, $\text{ker } F$ is contractible as well. 
\end{proof}

Since the dg nerve functor sends weak equivalences of dg categories to weak equivalences of quasi-categories ( Proposition 1.3.1.20 of \cite{Lu11})  we may deduce directly from Corollary \ref{equiv1} the following

\begin{corollary} For any conilpotent dg coalgebra $C$ the quasi-categories $N_{dg}  \text{Mod}^{\infty}_{\Omega C} $ and $N_{dg} \text{Mod}^{\tau}_{\Omega C}$ are weakly equivalent. 
\end{corollary}

As before, let $(X,b)$ a pointed path-connected space and $C=(C^N_*(\text{Sing}(X,b); \mathbf{k}), \partial, \Delta)$. 

\subsection{Equivalence between $Loc^{\infty}_X$ and $N_{dg}\text{Mod}^{\tau}_{\Omega C}$}\
\\

We first construct a functor of quasi-categories
\begin{equation*}
\theta: Loc^{\infty}_X \to N_{dg}\text{Mod}^{\tau}_{\Omega C}.
\end{equation*}
From the adjunction $(\Lambda, N_{dg})$, we obtain the following natural bijections for the sets of $n$-simplices:
\begin{equation*}
(Loc^{\infty}_X)_n \cong Set_{\Delta} (\Delta^n \times \text{Sing}(X,b), N_{dg}Ch_{\mathbf{k}}) \cong dgCat_{\mathbf{k}} (\Lambda(\Delta^n \times \text{Sing}(X,b)), Ch_{\mathbf{k}})
\end{equation*}
and
\begin{equation*}
(N_{dg}\text{Mod}^{\tau}_{\Omega C})_n \cong Set_{\Delta}(\Delta^n, N_{dg}\text{Mod}^{\tau}_{\Omega C}) \cong dgCat_{\mathbf{k}}(\Lambda(\Delta^n), \text{Mod}^{\tau}_{\Omega C}).
\end{equation*}
Hence, constructing $\theta$ is equivalent to constructing set maps 
\begin{equation*}
\theta: dgCat_{\mathbf{k}} (\Lambda(\Delta^n \times \text{Sing}(X,b)), Ch_{\mathbf{k}}) \to dgCat_{\mathbf{k}}(\Lambda(\Delta^n), \text{Mod}^{\tau}_{\Omega C})
\end{equation*}
for all integers $n \geq 0$, which are compatible with maps $[n] \to [m]$ in the ordinal category $\mathbf{\Delta}$. 

 Given any dg functor $F \in dgCat_{\mathbf{k}} (\Lambda(\Delta^n \times \text{Sing}(X,b)), Ch_{\mathbf{k}})$ define a dg functor $\theta(F): \Lambda(\Delta^n) \to \text{Mod}^{\tau}_{\Omega C}$ on objects $i \in \Lambda(\Delta^n)$ by letting $\theta(F)(i)$ be the chain complex $F(i,b)$ equipped with the right dg $\Omega C$-module given by pre-composing the map induced by $F$
 \begin{equation*}
F: \Lambda(\{i \} \times \text{Sing}(X,b))(b,b) \to Ch_{\mathbf{k}}(F(i,b), F(i,b))
\end{equation*}
with the isomorphism of dg algebras
\begin{equation*}
\Omega C \xrightarrow{\varphi} \Lambda(\text{Sing}(X,b))(b,b) \cong \Lambda(\{i \} \times \text{Sing}(X,b))(b,b)
\end{equation*}
 given in Theorem \ref{cobar}.  For simplicity, denote the right dg $\Omega C$-module $\theta(F)(i)$ by $F_i$. Given any two objects $i, j \in \Lambda(\Delta^n)$  we must define a chain map
\begin{eqnarray*}
\theta(F): \Lambda(\Delta^n)(i,j) \to \text{Mod}^{\tau}_{\Omega C}(F_i, F_j).
\end{eqnarray*}
First we introduce some notation. For any simplicial set $S$ and $i,j \in S_0$ let
\begin{equation} \label{EZ}
EZ: \Lambda(S)(i,j) \otimes C \to \Lambda( S \times \text{Sing}(X,b))( (i,b), (j,b))
\end{equation}
be the graded linear map defined on any generator $t \otimes \sigma \in \Lambda(S)(i,j) \otimes C$, where $(t: T \to S) \in  \Lambda(S)(i,j)$ and $(\sigma: \Delta^l \to \text{Sing}(X,b)) \in C$, 
by letting $EZ(t \otimes \sigma)$ be the (signed) sum of all generators in $\Lambda(S \times \text{Sing}(X,b)  )((i,b), (j,b))$ of degree $\text{dim}(T)+ \text{deg}(\sigma)$ with beads inside the sub-simplicial set $(t \times \sigma)(T \times \Delta^l) \subset S\times \text{Sing}(X,b)$.\footnote{$EZ$ stands for Eilenberg-Zilber since it extends the classical map, $C_*(X) \otimes C_*(Y) \to C_*(X \times Y)$, when both necklaces have one bead and the degree is correctly shifted.} In other words, if $T= \Delta^{n_1} \vee ... \vee \Delta^{n_k}$ then $EZ$ is given by considering all top dimensional necklaces inside $( \Delta^{n_1} \vee ... \vee \Delta^{n_k}) \times \Delta^l$ from vertex $(\alpha_T, 0)$ to vertex $(\omega_T, l)$ and mapping them to $S \times \text{Sing}(X,b)$ via $t \times \sigma$.

We now define $\theta(F)$ by applying the map $EZ$ when $S=\Delta^n$ as follows: for any generator $t \in \Lambda(\Delta^n)(i,j)$ as above let $\theta(F)(t) \in \text{Mod}^{\tau}_{\Omega C}( F_i, F_j )$ be the map of right $C$-comodules determined by the linear map $\theta(F)(t): F_i \otimes C \to F_j$ given by 
\begin{equation}
\theta(F)(t)(m \otimes \sigma) := (-1)^{|\sigma||m|} F( EZ (t \otimes \sigma ) )(m)
\end{equation}
for any generators $m \in F_i$ and $\sigma \in C$.  
\begin{proposition} The linear map $\theta(F): \Lambda(\Delta^n)(i,j) \to \text{Mod}^{\tau}_{\Omega C}(F_i, F_j)$ is a chain map.
\end{proposition}
\begin{proof}This follows from the (Maurer-Cartan type) formula
\begin{equation} \label{dbullet}
\begin{split}
&d_{\Lambda}  EZ(t \otimes \sigma)  = EZ( d_{\Lambda} t \otimes \sigma ) + (-1)^{|t|} EZ( t \otimes \partial' \sigma) \\
& - \sum_{ (\sigma) }  [ (i,\sigma' ) | EZ (t \otimes \sigma'')] + \sum_{ (\sigma) }(-1)^{|t| |\sigma'|} [EZ (t \otimes \sigma') | (j,  \sigma'') ],
\end{split}
\end{equation}
where $\partial'(\sigma):= \sum_{i=1}^{|\sigma|-1} (-1)^i d^{i} (\sigma)$, $\Delta(s)= \sum_{ (\sigma) } \sigma' \otimes \sigma '' \in C \otimes C$, and $(i,\sigma')$ denotes the generator in $\Lambda(\Delta^n \times \text{Sing}(X,b) ) ( (i,b), (i,b) )$ represented by the necklace with a single bead $\{ i \} \times \sigma'$, and $(j, \sigma '')$ is defined similarly. In fact, for any $m \in F_i$, \ref{dbullet} implies
\begin{equation} \label{Fdbullet}
\begin{split}
&F (d_{\Lambda} ( EZ (t \otimes \sigma ) )) (m)=  F( EZ ( d_{\Lambda}t \otimes \sigma) ) (m) \pm F( EZ( t \otimes \partial' \sigma ) ) (m)
\\ & \pm \sum_{(\sigma)} F(  EZ ( t \otimes \sigma'' ) ) ( \tau(\sigma') \cdot m) \pm  \sum_{(\sigma)} \tau(\sigma'') \cdot F( EZ(t \otimes  \sigma' ))(m) 
\end{split}
\end{equation}
and since $F: \Lambda(\Delta^n \times \text{Sing}(X,b))( (i,b), (j,b)) \to Ch_{\mathbf{k}}(F_i, F_j)$ is a chain map  we have
 \begin{equation} \label{Fchainmap}
 F(d_{\Lambda}(EZ(t \otimes \sigma)))(m)= d_{F_i} ( F ( EZ(t \otimes \sigma)  )(m) ) \pm F( EZ(t \otimes \sigma) ) (d_{F_i}m).
 \end{equation}
Combining \ref{Fdbullet} and \ref{Fchainmap} we obtain $\hat{\delta} (\theta(F)(t)) = \theta(F)(d_{\Lambda} t)$, where $\hat{\delta}$ is the differential as defined in \ref{deltahat}, as desired. 
\end{proof}

Since the classical Eilenberg-Zilber map is natural with respect to simplicial set morphisms it follows that the map $EZ$ of \ref{EZ} is natural with respect to morphisms $[n] \to [m]$ in $\mathbf{\Delta}$. It follows that the maps $\{ \theta: (Loc^{\infty}_X)_n \to (N_{dg} \text{Mod}^{\tau}_{\Omega C})_n \}$ are compatible with morphisms in $\mathbf{\Delta}$ as well, so they define a map of simplicial sets $\theta: Loc^{\infty}_X \to N_{dg} \text{Mod}^{\tau}_{\Omega C}$. 
\\
\\
We now show $\theta$ induces a homotopy equivalence of Kan complexes at the level of mapping spaces by constructing an explicit homotopy inverse. It is enough to show that for any two $\infty$-local systems $P,Q \in (Loc^{\infty}_X)_0 \cong dgCat_{\mathbf{k}}(\Lambda ( \text{Sing}(X,b) ), Ch_{\mathbf{k}})$, the morphism $\theta$ defined in Step 1 induces a homotopy  equivalence between right morphism spaces
\begin{equation}\label{thetamappingspaces}
\theta: Hom^R_{Loc^{\infty}_X}(P,Q) \to Hom^R_{N_{dg}\text{Mod}^{\tau}_{\Omega C}}(\theta(P), \theta(Q)).
\end{equation}
Recall that $Hom^R_{Loc^{\infty}_X}(P,Q) $ and $Hom^R_{N_{dg}\text{Mod}^{\tau}_{\Omega C}}(\theta(P), \theta(Q))$ are Kan complexes with sets of simplices given by 
\begin{equation*}
Hom^R_{Loc^{\infty}_X}(P,Q)_n\cong \{ f \in Set_{\Delta}(J^n, Loc^{\infty}_X) : f(x)= P, f(y)= Q \}
\end{equation*}
and
\begin{equation*}
Hom^R_{N_{dg}\text{Mod}^{\tau}_{\Omega C}}(\theta(P), \theta(Q))_n \cong \{ g \in Set_{\Delta}(J^n , N_{dg}\text{Mod}^{\tau}_{\Omega C}) : g(x)= \theta(P), g(y)=\theta(Q)\},
\end{equation*}
where $J^n$ is the simplicial set with two vertices $x$ and $y$ defined in section 2.5. Using the adjunction $(\Lambda, N_{dg})$ once again, the sets $Hom^R_{Loc^{\infty}_X}(P,Q)_n $ and $Hom^R_{N_{dg}\text{Mod}^{\tau}_{\Omega C}}(\theta(P), \theta(Q))_n$  described above correspond to
\begin{equation} \label{F}
 \{ F \in dgCat_{\mathbf{k} } ( \Lambda(J^n \times \text{Sing}(X,b) ) , Ch_{\mathbf{k}} ) : F|_{\Lambda( \{x\} \times \text{Sing}(X,b) )} =P, F|_{\Lambda( \{y\} \times \text{Sing}(X,b) )})=Q\}
\end{equation}
and 
\begin{equation} \label{G}
 \{ G \in dgCat_{\mathbf{k} } ( \Lambda( J^n ), \text{Mod}^{\tau}_{\Omega C} ) : G(x) = \theta(P), G(y)= \theta(Q) \},
\end{equation}
respectively. Note that the generators of $\Lambda(J^n)(x,y)$ correspond to necklaces with a single bead so they are determined by generators in $C^N_{*>0}(J^n)$ with a degree shift of $-1$. We use these identifications to construct a morphism of simplicial sets 

\begin{equation}\label{psi}
\psi: Hom^R_{N_{dg}\text{Mod}^{\tau}_{\Omega C}}(\theta(P), \theta(Q)) \to Hom^R_{Loc^{\infty}_X}(P,Q)
\end{equation}
which we will show to be a homotopy inverse for $\theta: Hom^R_{Loc^{\infty}_X}(P,Q) \to Hom^R_{N_{dg}\text{Mod}^{\tau}_{\Omega C}}(\theta(P), \theta(Q)).$

Given any $G \in dgCat_{\mathbf{k} } ( \Lambda( J^n ), \text{Mod}^{\tau}_{\Omega C} )$ such that $G(x) = \theta(P)$ and $G(y)= \theta(Q)$ define the dg functor $\psi(G) \in dgCat_{\mathbf{k}}( \Lambda(J^n \times \text{Sing}(X,b)) , Ch_{\mathbf{k}})$ on objects by setting $\psi(G)(x,b)=P(b)$ and $\psi(G)(y,b)=Q(b)$.
We proceed by defining chain maps
\begin{equation} \label{psiG}
\psi(G): \Lambda(J^n \times \text{Sing}(X,b) ) (z, z') \to Ch_{\mathbf{k}}(\psi(G)(z), \psi(G)(z')).
\end{equation}

If  $z=z'=(x,b)$ then any necklace $s: T \to J^n \times \text{Sing}(X,b)$ representing a generator in $\Lambda(J^n \times \text{Sing}(X,b) ) (z, z')$ lies inside $ \{x \} \times \text{Sing}(X,b)$. In this case define  $\psi(G)$ via the composition
$$\Lambda(\{x \} \times \text{Sing}(X,b)((x,b), (x,b))\cong \Lambda(\text{Sing}(X,b))(b,b) \xrightarrow{P} Ch_{\mathbf{k}}(P(b), P(b)).$$
The case $z=z'=(y,b)$ is similar.

Suppose now that $z=(x,b)$ and $z'=(y,b)$. Then any generator $s \in  \Lambda( J^n  \times \text{Sing}(X,b) )(z,z')$ may be represented by a necklace $s: \Delta^{n_1} \vee ... \vee \Delta^{n_r} \to J^n \times \text{Sing}(X,b)$ such that there exists some $1 \leq p \leq r$ for which $s$ maps $\Delta^{n_1} \vee ... \vee \Delta^{n_{p-1}}$ into $\{x \} \times \text{Sing}(X,b)$, $s$ maps the first and last vertices of $\Delta^{n_p}$ to $(x,b)$ and $(y,b)$ respectively, and $s$ maps $\Delta^{n_{p+1}} \vee ... \vee \Delta^{n_r}$ into $\{y \} \times \text{Sing}(X,b)$. For any $1 \leq l < k \leq p$, write $s_{l,k}$ for the restriction of $s$ to $\Delta^{n_l} \vee \Delta^{n_{l+1}} \vee ... \vee \Delta^{n_{k}}$ and $s_p$ for the restriction of $s$ to the bead $\Delta^{n_p}$. Note that $s_{1,{p-1}}$ may be considered as a generator in $\Lambda( J^n \times \text{Sing}(X,b) )( (x,b), (x,b) )$. Similarly, $s_{p+1,r}$ may be considered as a generator in $\Lambda( J^n \times \text{Sing}(X,b) )( (y,b), (y,b) ).$ 

For any such $s \in  \Lambda( J^n  \times \text{Sing}(X,b) )(z,z')$ define \ref{psiG} by
\begin{equation}
\psi(G)(s) := \sum_{ (s_p)  }  \psi(G)(s_{p+1,r}) \circ G([s_p' ]) (\_  \otimes s_p'' ) \circ \psi(G)(s_{1,p-1}),
\end{equation}
where we have written $\Delta(s_p)= \sum_{ (s_p) } s_p' \otimes s_p'' \in C^N_*(J^n) \otimes C^N_*(\text{Sing}(X,b))$ for the Alexander Whitney map applied to $s_p: \Delta^{n_p} \to J^n \times \text{Sing}(X,b)$. Similar arguments to those in Step 1 verify that the map $\psi(G): \Lambda(J^n \times \text{Sing}(X,b) ) (z, z') \to Ch_{\mathbf{k}}(\psi(G)(z), \psi(G)(z'))$ is a chain map and that $\psi(G)$ is natural with respect to maps $[n] \to [m]$ in $\mathbf{\Delta}$. Therefore, the above data defines a morphism $\psi: Hom^R_{N_{dg}\text{Mod}^{\tau}_{\Omega C}}(\theta(P), \theta(Q)) \to Hom^R_{Loc^{\infty}_X}(P,Q)$ of simplicial sets. 

\begin{proposition} \label{hptyinverse}
For any $P, Q \in (Loc^{\infty})_0$, the map $$\psi: Hom^R_{N_{dg}\text{Mod}^{\tau}_{\Omega C}}(\theta(P), \theta(Q)) \to Hom^R_{Loc^{\infty}_X}(P,Q)$$ is a homotopy inverse for $\theta: Hom^R_{Loc^{\infty}_X}(P,Q) \to Hom^R_{N_{dg}\text{Mod}^{\tau}_{\Omega C}}(\theta(P), \theta(Q))$. 
\end{proposition}
\begin{proof} Note that for any $G \in Hom^R_{N_{dg} \text{Mod}^{\tau}_{\Omega C}}(\theta(P), \theta(Q))_n$, $t \in \Lambda(J^n)(x,y)$, $ \sigma \in C$, and $m \in \theta(P)$ we have
\begin{equation} \label{comp1}
\theta(\psi(G))(t)(m \otimes \sigma) = \sum_{ (EZ(t \otimes \sigma)) } (-1)^{\epsilon} G ( [EZ(t \otimes \sigma )' ] )(m \otimes EZ(t \otimes \sigma) ''),
\end{equation}
where $\epsilon = |\sigma||m| + |EZ(t \otimes \sigma)''||m|$. Since $\sum_{(EZ(t \otimes \sigma))} EZ(t \otimes \sigma)' \otimes EZ(t \otimes \sigma)'' = t \otimes \sigma$ it follows that equation \ref{comp1} equals  $G ( t )(m \otimes \sigma)$. Thus, $\theta \circ \psi= id$. 

We now argue for the existence of a chain homotopy $\psi \circ \theta \simeq id$. Let $F$ be any element of the set $Hom^R_{Loc^{\infty}_X}(P,Q)_n$ (identified with \ref{F}) and $s \in \Lambda(J^n \times \text{Sing}(X,b))(z,z')$ be any generator represented by a map  $s: \Delta^{n_1} \vee ... \vee \Delta^{n_r} \to J^n \times \text{Sing}(X,b) )$ which sends $\Delta^{n_1} \vee ... \vee \Delta^{n_{p-1}}$ into $\{x \} \times \text{Sing}(X,b)$, $s$ sends the first and last vertices of the $p$-th bead $\Delta^{n_p}$ to $(x,b)$ and $(y,b)$ respectively, and $s$ sends $\Delta^{n_{p+1}} \vee ... \vee \Delta^{n_r}$ into $\{y \} \times \text{Sing}(X,b)$. Then we have

\begin{equation} \label{comp2}
\psi (\theta(F) )(s)= \sum_{(s_p)} F( [s_{1,p-1} | EZ( [s_{p}'] \otimes s_p'') | s_{p+1,r}]).
\end{equation}

An acyclic models argument (similar to the proof of the classical Eilenberg-Zilber theorem) yields that the natural chain map $$\Lambda(J^n \times \text{Sing}(X,b))(z,z') \to \Lambda(J^n \times \text{Sing}(X,b))(z,z')$$ given by $$s \mapsto \sum_{(s_p)} [s_{1,p-1} | EZ( [s_{p}'] \otimes s_p'') | s_{p+1,r}]$$
is chain homotopic to the identity map via a chain homotopy $$H: \Lambda(J^n \times \text{Sing}(X,b))(z,z') \to \Lambda(J^n \times \text{Sing}(X,b))(z,z').$$

We now construct a simplicial homotopy 
 \begin{equation}
 h: Hom^R_{Loc^{\infty}_X}(P,Q) \times \Delta^1 \to Hom^R_{Loc^{\infty}_X}(P,Q)
 \end{equation} 
 yielding $ \psi \circ \theta \simeq id$. For any simplex $F \in Hom^R_{Loc^{\infty}_X}(P,Q)_n$ consider the prism $F \times \Delta^1= \bigcup_{i=0}^n S^F_i$ where each $S^F_i \in (Hom^R_{Loc^{\infty}_X}(P,Q) \times \Delta^1)_{n+1} $ is the image under $$F \times id: \Delta^n \times \Delta^1 \to (Hom^R_{Loc^{\infty}_X}(P,Q) \times \Delta^1)_{n+1}$$ of the $i$-th $(n+1)$-simplex $S_i \in (\Delta^n \times \Delta^1)_{n+1}$ in the usual subdivision of a prism into simplices. The homotopy $h$ is determined by defining
\begin{equation}
h (S^F_i ): \Lambda(J^{n} \times \text{Sing}(X,b)) (z,z')  \to Ch_{k}(F(z),F(z'))
\end{equation}
as $h (S^F_i)=0$ if $i \neq 0$ and $h (S^F_0)=  F \circ H$, where $F \circ H$ denotes the composition
$$ \Lambda(J^n \times \text{Sing}(X,b))(z,z') \xrightarrow{H}  \Lambda(J^n \times \text{Sing}(X,b))(z,z') \xrightarrow{F} Ch_{\mathbf{k}}(F(z), F(z')).$$
\end{proof}

\begin{theorem}\label{main} The functor $\theta: Loc^{\infty}_X \to N_{dg}\text{Mod}^{\tau}_{\Omega C}$ is a weak equivalence of quasi-categories.
\end{theorem}

\begin{proof}The adjunction $(\Lambda, N_{dg})$ implies that a right dg module over the dg algebra $$ \Omega C \cong \Lambda(\text{Sing}(X,b))(b,b)$$ is equivalent to a morphism of simplicial sets $\text{Sing}(X,b) \to N_{dg}Ch_{\mathbf{k}}$, so $\theta: Loc^{\infty}_X \to N_{dg} \text{Mod}^{\tau}_{\Omega C}$ induces an essentially surjective functor of homotopy categories. It follows from Proposition \ref{hptyinverse} that $\theta: Loc^{\infty}_X \to N_{dg} \text{Mod}^{\tau}_{\Omega C}$ induces a homotopy equivalence of Kan complexes at the level of mapping spaces. 
\end{proof}

\begin{remark} There are two other models for the $\infty$-category $N_{dg}Mod^{\tau}_{\Omega C}$ in the literature. In \cite{BlSm14} it is shown that, when $X$ is a manifold and $C$ is the dg coalgebra of smooth chains on $X$, the dg category $Mod^{\tau}_{\Omega C}$ is $A_{\infty}$-quasi-equivalent to a dg category of smooth $\mathbb{Z}$-graded vector bundles over $X$ with homotopy-coherently flat super connection (with certain finiteness assumptions). This statement generalizes the classical Riemann-Hilbert correspondence. It is explained in section 8.4 of \cite{Po11} that for any conilpotent dg coalgebra $C$ there is a model structure on the category of dg $C$-comodules which is Quillen equivalent to the standard model structure on the category of dg $\Omega C$-modules. This is a manifestation of Koszul duality. A Quillen equivalence of model categories induces an equivalence of the associated $\infty$-categories. The $\infty$-category associated to the model category of dg $\Omega C$-modules is equivalent to $N_{dg}Mod^{\tau}_{\Omega C}$. 
\end{remark}

\section{The colimit of an $\infty$-local system}

We use Theorem \ref{main} to obtain a small model for the colimit of an $\infty$-local system. We recall the definition and a criterion given in \cite{Lu09} for the (homotopy) colimit of  $\beta: K \to \mathcal{C}$, where $K$ is a simplicial set, $\mathcal{C}$ a quasi-category, and $\beta$ a map of simplicial sets. We first recall some notation. For any two simplicial sets $K$ and $L$ denote by $K \star L$ the simplicial set whose set of $n$-simplices is given by 
\begin{equation*}
(K \star L)_n= K_n \cup L_n \cup \bigcup_{i+j=n-1}K_i \times L_j
\end{equation*}
with structure maps induced by those in $L$ and $K$. In particular, we call $K \star \Delta^0$ the \textit{right cone on K}. Define $\mathcal{C}_{\beta/}$ to be the simplicial set whose set of $n$-simplices is 
\begin{equation*}
(\mathcal{C}_{\beta/})_n=\text{Hom}_{\beta}( K \star \Delta^n, \mathcal{C}),
\end{equation*}
the set of simplicial set maps $K \star \Delta^n \to \mathcal{C}$ extending $\beta: K \to \mathcal{C}$. By Proposition 1.2.9.3 in \cite{Lu09}, if $\mathcal{C}$ is a quasi-category then so is $\mathcal{C}_{\beta/}$. An object $x$ in a quasi-category $\mathcal{D}$ is called \textit{initial} if the right hom space $Hom^R_{\mathcal{D}}(x,y)$ is a contractible Kan complex for all objects $y$ in $\mathcal{D}$. 
\begin{definition} 
Given any map of simplicial sets $\beta: K \to \mathcal{C}$ where $\mathcal{C}$ is a quasi-category, a \textit{colimit} for $\beta$ is defined to be an initial object of $\mathcal{C}_{\beta/}$. Note that an object of $\mathcal{C}_{\beta/}$ can be identified with a map $\overline{\beta}: K \star \Delta^0 \to \mathcal{C}$ extending $\beta$. We sometimes abuse notation and refer to $\overline{\beta}(*) \in \mathcal{C}$ as the colimit of $\beta$, where $*$ denotes the cone point of $K \star \Delta^0$. 
\end{definition}
We have the following criterion (Lemma 4.2.4.3 in \cite{Lu09}) for detecting colimits in quasi-categories:
\begin{lemma}\label{crit}
Let $\mathcal{C}$ be a quasi-category, $K$ a simplicial set and $\overline{\beta}: K \star \Delta^0 \to \mathcal{C}$ a map of simplicial sets. Then the following conditions are equivalent:
\begin{itemize}
\item $\overline{\beta}: K \star \Delta^0 \to \mathcal{C}$ is a colimit of $\beta=\overline{\beta}|_{K}$
\item Let $\delta: \mathcal{C} \to \text{Fun}(K, \mathcal{C})$ denote the diagonal embedding, and let $\alpha: \beta \to \delta( \overline{\beta}(*) )$ denote the natural transformation determined by $\overline{\beta}$. Then, for every object $Y \in \mathcal{C}$, composition with $\alpha$ induces a homotopy equivalence
\begin{equation*}
\text{Hom}^R_{\mathcal{C} } (\overline{\beta}(*), Y) \to \text{Hom}^R_{\text{Fun}(K,\mathcal{C})}(\beta, \delta(Y)).
\end{equation*}
\end{itemize}
\end{lemma}

We give a model for the colimit of an $\infty$-local system of chain complexes in terms of the twisted tensor product construction. Let $(X,b)$ be a path-connected pointed space and $\beta: \text{Sing}(X,b) \to N_{dg}Ch_{\mathbf{k}}$ be an $\infty$-local system over $X$. Denote $(M, d_M)=\beta(b)$ and $C=(C^N_*(\text{Sing}(X,b)), \partial, \Delta)$, so that $M$ is a right dg $\Omega C$-module. Define a map $\overline{\beta}: \text{Sing}(X,b) \star \Delta^0 \to N_{dg}Ch_{\mathbf{k}}$ which extends $\beta$, by setting $\overline{\beta}(*):= (M \otimes C, \partial_{\otimes_{\tau}})$, and for any simplex $\sigma \in (\text{Sing}(X,b) \star \Delta^0)_n$ such that $n>0$ and the first and last vertices of $\sigma$ are $b$ and $*$ respectively, let $(\overline{\beta}(\sigma): (M, d_M) \to (M \otimes C, \partial_{\otimes_{\tau}})) \in (N_{dg}Ch_{\mathbf k})_{n}$ be the map of degree $n-1$ defined by $\overline{\beta}(\sigma)(m)= m \otimes \partial_n(\sigma) $, where $\partial_n(\sigma)$ is the last face of $\sigma$ (the face opposite to $*$, which lies inside $\text{Sing}(X,b)$). In particular, for any $1$-simplex $\sigma$ from $b$ to $*$, the $1$-simplex $\overline{\beta}(\sigma): (M,d_M) \to (M \otimes C, \partial_{\otimes_{\tau}})$ in $(N_{dg}Ch_{\mathbf{k}})_1$ is the inclusion chain map $m \mapsto m \otimes \sigma_b$ where $\sigma_b$ is the unique generator of $C_0$. This yields a map of simplicial sets $\overline{\beta}: \text{Sing}(X,b) \star \Delta^0 \to N_{dg}Ch_{\mathbf{k}}$. 

\begin{theorem} \label{colimprop} For any $\infty$-local system $\beta: \text{Sing}(X,b) \to N_{dg}Ch_{\mathbf{k}}$ the map $\overline{\beta}: \text{Sing}(X,b) \star \Delta^0 \to N_{dg}Ch_{\mathbf{k}}$ constructed above is a colimit of $\beta$. 
\end{theorem}
\begin{proof} The result follows from Lemma \ref{crit} with $K=\text{Sing}(X,b)$ and $\mathcal{C}=N_{dg}Ch_{\mathbf{k}}$ so that $\text{Fun}(K, \mathcal{C})= Loc_X^{\infty}$, together with the  following two observations:
\\
1) for any chain complex $(Y,d_Y) \in Ch_{\mathbf{k}}$ with the trivial dg $\Omega C$-module structure
\begin{equation*}
 Hom^R_{N_{dg} Ch_{\mathbf{k}}}(   (M \otimes C, \partial_{\otimes_{\tau}}) , (Y,d_Y)) \cong Hom^R_{N_{dg} Mod^{\tau}_{\Omega C}} ( (M, d_M)  , (Y,d_Y)),
\end{equation*}
and
\\
2) by the proof of Theorem \ref{main}, the natural map of Kan complexes described in \ref{psi}
\begin{equation*}
Hom^R_{N_{dg} Mod^{\tau}_{\Omega C}}( (M,d_M), (Y,d_Y) ) \to Hom^R_{Loc^{\infty}_X}(\beta , \delta(Y,d_Y))
\end{equation*}
is a homotopy equivalence and the compositions of maps 1) and 2) above coincides with the homotopy equivalence induced by the natural transformation $\alpha: \beta \to \delta( \overline{\beta}(*) )$ as described in the second item of Lemma \ref{crit}.
\end{proof} 

\section{Recovering a classical result of Brown}

For any fibration $\pi: E\to X$ over a path-connected space $X$, the path lifting property induces a dg $C_*(\Omega_bX)$-module structure on $C_*(F)$, the singular chains on the fiber $F=\pi^{-1}(b)$. Hence, $C_*(F)$ becomes dg $\Omega C$-module where $C=(C^N_*(\text{Sing}(X,b) ), \partial, \Delta)$. An explicit chain map $\Upsilon: \Omega C \to C^{\square}_*(\Omega_b X)$ is given in \cite{Ad52}, where  $C^{\square}_*(\Omega_b X)$  denotes the normalized singular cubical chains on $\Omega_bX$. Adams defines $\Upsilon$ by constructing a collection of maps $\{ \upsilon_n: [0,1]^{n-1} \to P_{0,n}|\Delta^n|\}$, where $P_{0,n}|\Delta^n|$ is the space of Moore paths in the topological $n$-simplex $|\Delta^n|$ from vertex $0$ to vertex $n$. The maps $\{ \upsilon_n \}$ satisfy a compatibility equation that relates the cubical faces of $[0,1]^{n-1}$ to the simplicial faces and the Alexander-Whitney coproduct terms on $\Delta^n$ and implies that $\Upsilon$ is a chain map.

We give a more conceptual proof of the main result of \cite{Br59} which says that the chains on the total space of a fibration may be modeled as a twisted tensor product between chains on the fiber and chains on the base. 

\begin{theorem}\cite{Br59}
Let $\pi: E \to X$ be a fibration over a path-connected space $X$ with $F = \pi^{-1}(b)$ for $b \in X$. Then, there is weak equivalence of chain complexes 
\begin{equation}
C_*(E) \simeq ( C_*(F) \otimes C, \partial_{\otimes_\tau}).
\end{equation}
\end{theorem}
\begin{proof}

Let $\mathcal{S}$ be the $\infty$-category of spaces and and $\mathcal{S}_{/X} $ the $\infty$-category of spaces over $X$, both thought of as quasi-categories. Recall that by the $\infty$-categorical version of the Grothendieck construction there is a weak equivalence of quasi-categories 
\begin{equation}
\varphi: \mathcal{S}_{/X} \simeq \text{Fun}(\text{Sing}(X,b), \mathcal{S}).
\end{equation}
Hence, any fibration $\pi: E \to X$ gives rise to a functor of quasi-categories $\varphi_{\pi}: \text{Sing}(X,b) \to \mathcal{S}$ and $\text{colim } \varphi_{\pi} \simeq E$, where $\text{colim}$ denotes the (homotopy) colimit as defined in the section 6. The fact that $\text{colim } \varphi_{\pi} \simeq E$ follows from Proposition 2.1 of the exposition \cite{Ma-Ge15}; more details may be found in section 3.3.4 of \cite{Lu09}. 

An explicit description of $\varphi_{\pi}$ may be obtained as follows: for $b \in  \text{Sing}(X,b)_0$ we have $\varphi_{\pi}(b) =\pi^{-1}(b) = F$, for any $1$-simplex $\gamma \in \text{Sing}(X,b)_1$ the path lifting property provides a homotopy equivalence $\varphi_{\pi}(\gamma): F \to F$, and in general, for any $\sigma \in \text{Sing}(X,b)_n$ the path lifting property applied to the family of paths $\upsilon_n: [0,1]^{n-1} \to P_{0,n}|\Delta^n|$ constructed by Adams provides a continuous map $\varphi_{\pi}(\sigma): F \times [0,1]^{n-1} \to F$. 

After taking singular cubical chains over $\mathbf{k}$, $\varphi_{\pi}$ gives rise to an $\infty$-local system $\beta_{\pi} : \text{Sing}(X,b) \to N_{dg}Ch_{\mathbf{k}}$, for which $\beta_{\pi}(b)= C_*(F)$, and since the singular chains functor preserves (homotopy) colimits we obtain $\text{colim } \beta_{\pi} \simeq C_*(E)$. It follows from Theorem \ref{colimprop} that $C_*(E) \simeq ( C_*(F) \otimes C, \partial_{\otimes_\tau})$, as desired.

\end{proof}

\bibliographystyle{plain}

 \Addresses

\end{document}